\documentclass[
a4paper,reqno
]{amsart}

\usepackage{fullpage}

\usepackage[english]{babel}
\usepackage[utf8]{inputenc}

\usepackage{enumitem}
\setlist[enumerate]{label=\text{(\roman*)}}

\numberwithin{equation}{section}

\usepackage{mathtools}
\mathtoolsset{showonlyrefs}
\usepackage{physics}
\usepackage{xfrac}

\usepackage{hyperref}

\allowdisplaybreaks
\usepackage{graphicx}
\usepackage{tikz}
\usetikzlibrary{snakes}
\usetikzlibrary{arrows.meta,decorations.pathreplacing}
\usepackage{pgfplots}
\usepackage{subcaption}

\makeatletter
\newcommand{\labeltarget}[1]{\Hy@raisedlink{\hypertarget{#1}{}}}
\makeatother

\newcommand{\bbR}{\mathbb{R}}
\newcommand{\bbN}{\mathbb{N}}
\newcommand{\bbQ}{\mathbb{Q}}
\newcommand{\bbZ}{\mathbb{Z}}

\newcommand{\bbS}{\mathbb{S}}

\DeclareMathOperator{\range}{range}
\newcommand{\sft}{\mathsf{t}}
\newcommand{\sfN}{\mathsf{N}}
\newcommand{\sfn}{\mathsf{n}}
\newcommand{\rmC}{\mathrm{C}}
\newcommand{\rmH}{\mathrm{H}}

\DeclareMathOperator{\vol}{vol}
\DeclareMathOperator{\len}{len}

\newcommand{\ODE}{\textsc{ode}{}}

\usepackage{theoremref}

\newtheorem{theo}{Theorem}[section]
\newtheorem{satz}[theo]{Proposition}
\newtheorem{lemma}[theo]{Lemma}

\newtheorem{coro}[theo]{Corollary}

\newtheorem{main}{Main Result}

\newtheorem{lemma_no_break}[theo]{Lemma}

\theoremstyle{definition}
\newtheorem{defi}[theo]{Definition}
\newtheorem{nota}[theo]{Notation}
\newtheorem{ex}[theo]{Example}
\newtheorem{dis}[theo]{Disclosure}
\newtheorem{rem}[theo]{Remark}

\newcommand\inner[2]{\left\langle #1, #2 \right\rangle}

\usepackage{mdframed}
\begin{document}
	
	\title[Shrinkers of the area-preserving curve-shortening flow]{Shrinkers of the area-preserving curve-shortening flow:\\ Classification and saddle-point property}
	
	\author[Nikita Cernomazov]{Nikita Cernomazov}

	\maketitle
	
		\begin{mdframed}
		\begin{abstract}
			We consider homothetic evolutions of the area-preserving curve-shortening flow (APCSF), that is, classical curve-shortening flow with an additional non-local forcing term. By using known results on $\lambda$-curves, we prove the existence of non-circular shrinkers for this flow. In our first main result, we present a full classification of such curves, similar to the well-known Abresch--Langer classification for shrinkers of curve-shortening flow. Finally, we also deduce a saddle-point property for all non-circular (APCSF)-shrinkers analogous to the known saddle-point property of Abresch--Langer curves.
		\end{abstract}
	\end{mdframed}
	
	\tableofcontents
	
	\section{Introduction}
	
	Suppose that $X:M^1\times[0,T)\to\bbR^2$ is a $1$-parameter family of smooth immersions with $T\in(0,\infty]$ and either $M^1\cong \bbS^1$ or $M^1\cong\bbR$. 
	For each ``time'' $t\in[0,T)$, we write $k(\cdot,t)$ for the curvature with respect to the unit normal field $\sfN(\cdot,t)$ obtained from a rotation of the unit tangent field by $+\frac{\pi}{2}$. Note that in the case of e.g.~positively oriented embeddings of $\bbS^1$, our choice of normal directions $\sfN$ corresponds to \emph{inward}-pointing vectors.
	Denoting $\dot{X}(\cdot,t)\coloneqq\partial_t(X(\cdot,t))$ as the \emph{time derivative}, we say that $X$ (or $X(M^1,\cdot)$) evolves by \emph{curve-shortening flow}, \eqref{CSF}, if
	\begin{equation}
		\dot{X}=k\sfN\tag{\textsc{csf}}\quad\text{on $M^1\times[0,T)$}.\label{CSF}
	\end{equation}
	Up to ``spatial'' reparametrizations of $X$, the above equation holds if and only if $\langle\dot{X},\sfN\rangle=k$ on $M^1\times[0,T)$. Thus, \eqref{CSF} is often interpreted as the (negative) $\mathrm{L}^2$-gradient flow of the length functional,
	\begin{equation}
		\len(t)\coloneqq\len[X(\cdot,t)]=\int_{M^1}\dd{s(\cdot,t)}
	\end{equation}
	with $s(\cdot,t)$ being the arc-length element corresponding to $X(\cdot,t)$ for $t\in[0,T)$. This evolution equation is well studied with several strong results available (see e.g.~\cite[§2--§4]{AndChoGue00} for a detailed overview). 
	Now suppose in the sequel that $M^1\cong\bbS^1$ (or equivalently, that $M^1\cong \bbR$ and $X(\cdot,t)$ is periodic for all $t\in[0,T)$). Then, the (\emph{signed} or \emph{algebraic}) \emph{area} enclosed by $X(\cdot,t)$, i.e.
	\begin{equation}
		\vol(t)\coloneqq\vol[X(\cdot,t)]=-\frac{1}{2}\int_{M^1}\inner{X(\cdot,t)}{\sfN(\cdot,t)}\dd{s(\cdot,t)}\label{area}
	\end{equation}
	evolves by $\dot{\vol}(t)=-2\pi m$ with $m\in\bbZ$ being the tangent turning index of $X(\cdot,t)$. Hence, the enclosed \emph{geometric} area $\abs{\vol}$ of e.g.~an embedding ($m=\pm 1$) strictly decreases along the flow. A modification of \eqref{CSF} that forces $t\mapsto\vol(t)$ to be constant along the flow may be constructed by including an appropriate Lagrange multiplier. Following this idea, \cite{Gag86} found that subtracting the \emph{average curvature}
	\begin{equation}
		\bar{k}(t)\coloneqq\bar{k}[X(\cdot,t)]\coloneqq \frac{1}{\len(t)}\int_{M^1}k(\cdot,t)\dd{s(\cdot,t)}=\frac{2\pi m}{\len(t)}
	\end{equation}
	``constrains'' the gradient flow $t\mapsto X(\cdot,t)$ onto trajectories of constant algebraic area. Indeed, one can easily check that evolutions by \emph{area-preserving curve-shortening flow}, \eqref{APCSF},
	\begin{equation}
		\dot{X}=(k-\bar{k})\mathsf{N}\quad\text{on $M^1\times [0,T)$}\tag{\textsc{apcsf}}\label{APCSF}
	\end{equation}
	satisfy both $\dot{\len}\leq 0$ and $\dot{\vol}=0$, thus justifying its name. Past studies pertaining to this evolution equation have focused on  typical questions related to curve flows. There are works on long-time behavior \cite{Gag86,WanKon14,SesTsaWan20}, the formation of singularities \cite{AnaIshUsh25} and the (non-)preservation of geometric properties like convexity \cite{Dit20} and embeddedness \cite{May01}. There are also results available for \eqref{APCSF} with boundary conditions \cite{Mad15, Mad18}.
	\begin{rem}[beyond \eqref{APCSF}]
		Other than \eqref{APCSF}, there are several other modifications to \eqref{CSF} that preserve the enclosed area (e.g.~\cite{EscIto05,Fan20,Lan25}). Moreover, the procedure described above is also applicable to the higher-dimensional formulation of \eqref{CSF}, namely \emph{mean curvature flow},
		\begin{equation}
			\dot{X}=\rmH \sfN\quad\text{on $M^d\times [0,T)$,}\tag{\textsc{mcf}}\label{MCF}
		\end{equation}
		where $X:M^d\times[0,T)\to\bbR^{d+1}$ is a family of e.g.~closed embedded hypersurfaces and $\sfN$ and $\mathrm{H}$ are the corresponding families of inward-pointing normal vectors and mean curvatures. By analogously subtracting an appropriate global quantity in the evolution equation above, one may define a \textit{volume-preserving} mean curvature flow. Finally, by choosing non-local forcing terms other than $\bar{k}$ one can also impose the preservation of other quantities along the flow.  An outline describing recent developments regarding such \textit{constrained} flows of curves and surfaces can be found in \cite{Cab23}.
	\end{rem}
	\subsection{Homothetic evolutions}
	The present work is chiefly concerned with immersions $x:M^1\to \bbR^2$ that generate an evolution by homothety. Let us first formally define these.
	
	\begin{defi}[homothetic evolution]\thlabel{def_hom}
		Let $X:M^1\times [0,T)\to\bbR^2$ be a family of immersions. We say that $X$ (or $X(M^1,\cdot)$) \emph{moves} or \emph{evolves by homothety} if there is a \emph{scaling function} $\psi\in\mathrm{C}^1\qty([0,T);(0,\infty))$ such that
		\begin{equation}
			X(M^1,t)=x_\ast+\psi(t)\qty(X(M^1,0)-x_\ast)\quad\text{for all}\quad t\in[0,T)
		\end{equation}
		for some $x_\ast\in\bbR^2$. If additionally $\dot{\psi}\leq 0$, we call $X(\cdot,0)$ (or $X(M^1,0)$) a \emph{shrinker}.
	\end{defi}
	Clearly, any circular immersion generates a homothetic evolution by both \eqref{CSF} and \eqref{APCSF}. An early result found independently by \cite{AbrLan86} and \cite{EpsWei87} classifies all compact non-circular \eqref{CSF}-shrinkers.
	\hypertarget{abresch-langer}{
		\begin{theo}[classification of closed \eqref{CSF}-shrinkers; \cite{AbrLan86,EpsWei87}]\thlabel{abresch-langer}
			Up to similarity and multiple covers, the closed \eqref{CSF}-shrinkers are precisely $\bbS^1$ and, for each coprime $m,n\in\bbN$ verifying
			\begin{equation}
				\frac{1}{2}<\frac{m}{n}<\frac{\sqrt{2}}{2},\label{ineq2}
			\end{equation}
			a unique non-circular, strictly locally convex curve with tangent turning index $m$ whose image exhibits $n$-fold rotational symmetry.
	\end{theo}}
	Collectively, the non-circular \eqref{CSF}-shrinkers are called \emph{Abresch--Langer curves} and play a crucial role e.g.~in the analysis of singularities (see e.g.~\cite{Ang91}). Another interesting interpretation of such curves is in regard to their behavior under small perturbations realized via \textit{offset curves}. To make this rigorous, we introduce appropriate notation.
	\begin{nota}[normal perturbations]
		For a strictly locally convex immersion $x:M^1\to\bbR^2$ with corresponding inward-pointing unit normal vector $\sfn:M^1\to\bbS^1$, we write
		\begin{equation}
			x^{\pm\varepsilon}\coloneqq x\mp \varepsilon\sfn,
		\end{equation} 
		where $\varepsilon>0$ is chosen so small that $x^{\pm \varepsilon}$ remains an immersion.
	\end{nota}
	In response to a conjecture stated in \cite{AbrLan86}, \cite{Au10} found that Abresch--Langer curves act as \textit{saddle points} or \textit{watersheds} between curves that asymptotically converge to multiple covers of a circle and curves that develop singular cusps. The Abresch--Langer curves ``separate'' evolutions with \textit{singular} and \textit{circular} futures in the sense explained below.
	\begin{theo}[saddle-point property of \eqref{CSF}-shrinkers; \cite{Au10}]
		Let $x:\bbS^1\to\bbR^2$ be a non-circular $n$-symmetric \eqref{CSF}-shrinker with tangent turning index $m$. If, for $\varepsilon>0$ sufficiently small, $X^{\pm\varepsilon}:\bbS^1\times[0,T_\ast^\pm)\to\bbR^2$ is the maximal \eqref{CSF}-evolution with $X^{\pm\varepsilon}(\cdot,0)=x^{\pm\varepsilon}$, then
		\begin{equation}
			\frac{X^{\pm\varepsilon}(\cdot,t)}{\sqrt{2(T_\ast^\pm-t)}}\xrightarrow[\text{in $\mathrm{C}^\infty$}]{t\nearrow T_\ast^\pm}\begin{cases}
				\text{an $m$-fold cover of a circle},&\text{for ``$+$''},\\
				\text{a singular curve with $n$ cusps},&\text{for ``$-$''}.
			\end{cases}
		\end{equation}
		\thlabel{saddle}
	\end{theo}
	\begin{figure}[htbp]
		\centering
		\resizebox{\textwidth}{!}{\begin{tikzpicture}[
				arrow/.style={->, >=stealth, semithick},
				every node/.style={anchor=center}
				]
				
				% Die SVG-Grafiken als TikZ-Nodes einbinden
				\node at (-6.75,0) {\includegraphics[width=0.1\linewidth]{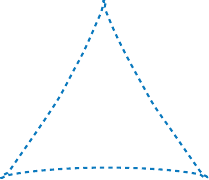}};
				\node at (-2.5,0)  {\includegraphics[width=0.1\linewidth]{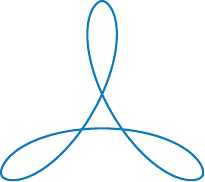}};
				\node at (0,0)  {\includegraphics[width=0.1\linewidth]{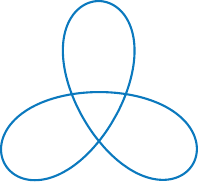}};
				\node at (2.5,0)  {\includegraphics[width=0.1\linewidth]{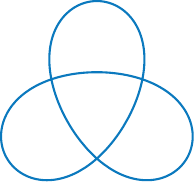}};
				\node at (6.75,0)  {\includegraphics[width=0.1\linewidth]{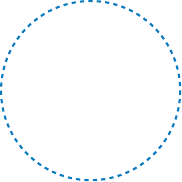}};
				
				% Die Pfeile kannst du dann wie gewohnt mit TikZ dazwischen zeichnen
				\draw [->,snake=snake,
				line after snake=1mm] (0.75, 0.25) -- (1.75, 0.25) node[midway, above] {$+\varepsilon \sfn$};
				\draw [->,snake=snake,
				line after snake=1mm] (-0.75,0.25) -- (-1.75, 0.25) node[midway, above] {$-\varepsilon \sfn$};
				\draw [->] (-3.5,0.25) -- (-5.75, 0.25)
				node[midway, above] {$T_\ast^- \nwarrow t$}
				node[midway, below] {asymptotically};
				\draw [->] (3.5,0.25) -- (5.75, 0.25)
				node[midway, above] {$t \nearrow T_\ast^+$}
				node[midway, below] {asymptotically};
		\end{tikzpicture}}
		\caption{Illustration of the saddle-point property of Abresch--Langer curves.}
	\end{figure}
	\begin{rem}[beyond Abresch--Langer]
		We further alert the reader that relaxing the compactness condition in \thref{abresch-langer} drastically expands the collection of \eqref{CSF}-shrinkers. Apart from the trivial example of a stationary line, the members of an uncountable class of curves whose images are dense in an annulus constitute examples of non-compact shrinkers. Moreover, in this case, there are also \eqref{CSF}-expanders (i.e., $\dot{\psi}\geq 0$), which are clearly impossible if $M^1\cong\bbS^1$. Such curves along with self-similar motions other than homotheties have been classified in \cite{Hal12}.
	\end{rem}
	\subsection{Results}
	As far as it is known to the author, there are no results on \eqref{APCSF}-shrinkers available. The first main result therefore presents a classification of such curves. It can be read as a statement analogous to \thref{abresch-langer}, although, somewhat surprisingly, comparing \eqref{ineq2} with \eqref{ineq} indicates that in a sense \eqref{APCSF} has ``more'' shrinkers than \eqref{CSF}. 
	\begin{main}[classification of \eqref{APCSF}-shrinkers]
		\thlabel{main-thm}
		Up to similarity and multiple covers, the closed \eqref{APCSF}-shrinkers are precisely $\bbS^1$ and, for each coprime $m,n\in\bbN$ with
		\begin{equation}
			\frac{1}{2}<\frac{m}{n}<1,\label{ineq}
		\end{equation}
		a unique non-circular, strictly locally convex curve with tangent turning index $m$ whose image exhibits $n$-fold rotational symmetry.
	\end{main}
	\begin{rem}[area of \eqref{APCSF}-shrinkers]\thlabel{zeroarea}
		Consider a shrinker with scaling function $\psi(\cdot)$. Then the enclosed area $\vol(\cdot)$ scales by $\psi^2(\cdot)$. Thus, because of area-preservation, any \eqref{APCSF}-shrinker is either stationary (i.e., circular) or encloses an area of zero.
	\end{rem}
	For our second result presented in this work, we consider the evolution of normally perturbed \eqref{APCSF}-shrinkers. We find their behavior to be completely analogous to the \eqref{CSF}-case.
	\begin{main}[saddle-point property of \eqref{APCSF}-shrinkers]
		\thlabel{main-thmB}
		Let $x:\bbS^1\to\bbR^2$ be a non-circular $n$-symmetric \eqref{APCSF}-shrinker with tangent turning index $m$. If, for $\varepsilon>0$ sufficiently small, $X^{\pm\varepsilon}:\bbS^1\times[0,T_\ast^\pm)\to\bbR^2$ is the maximal \eqref{APCSF}-evolution with $X^{\pm\varepsilon}(\cdot,0)=x^{\pm\varepsilon}$, then
		\begin{equation}
			X^{\pm\varepsilon}(\cdot,t)\xrightarrow[\text{in $\rmC^\infty$}]{t\nearrow T_\ast}\begin{cases}
				\text{an $m$-fold cover of a circle},&\text{for ``$+$''},\\
				\text{a singular curve with $n$ cusps},&\text{for ``$-$''}.
			\end{cases}
		\end{equation}
	\end{main}
	\begin{figure}[h]
		\centering
		\begin{subfigure}{0.24\textwidth}
			\includegraphics[width=\textwidth]{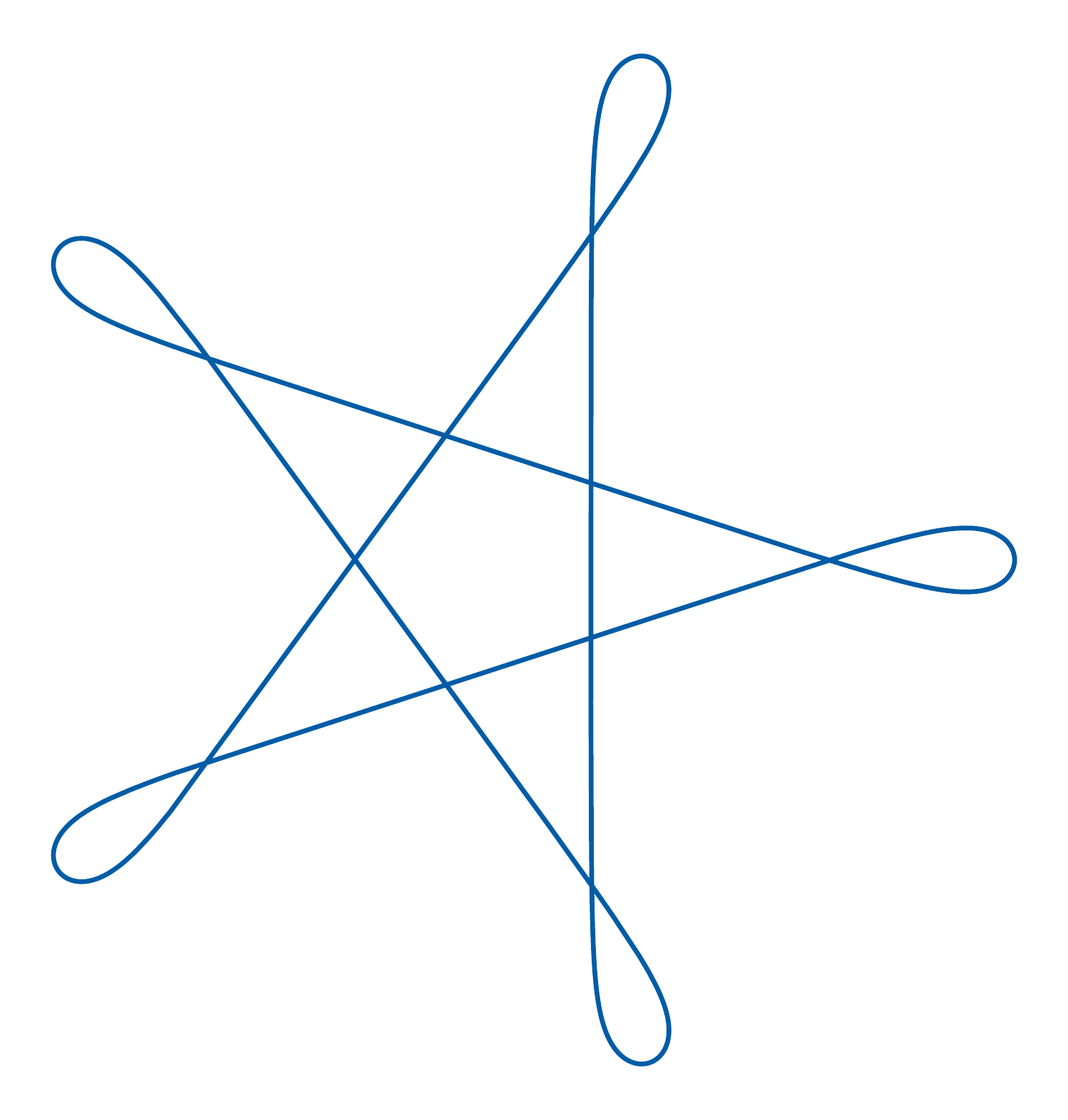}
			\caption*{$\sfrac{m}{n}=\sfrac{3}{5}$}
		\end{subfigure}
		\hfill
		\begin{subfigure}{0.24\textwidth}
			\includegraphics[width=\textwidth]{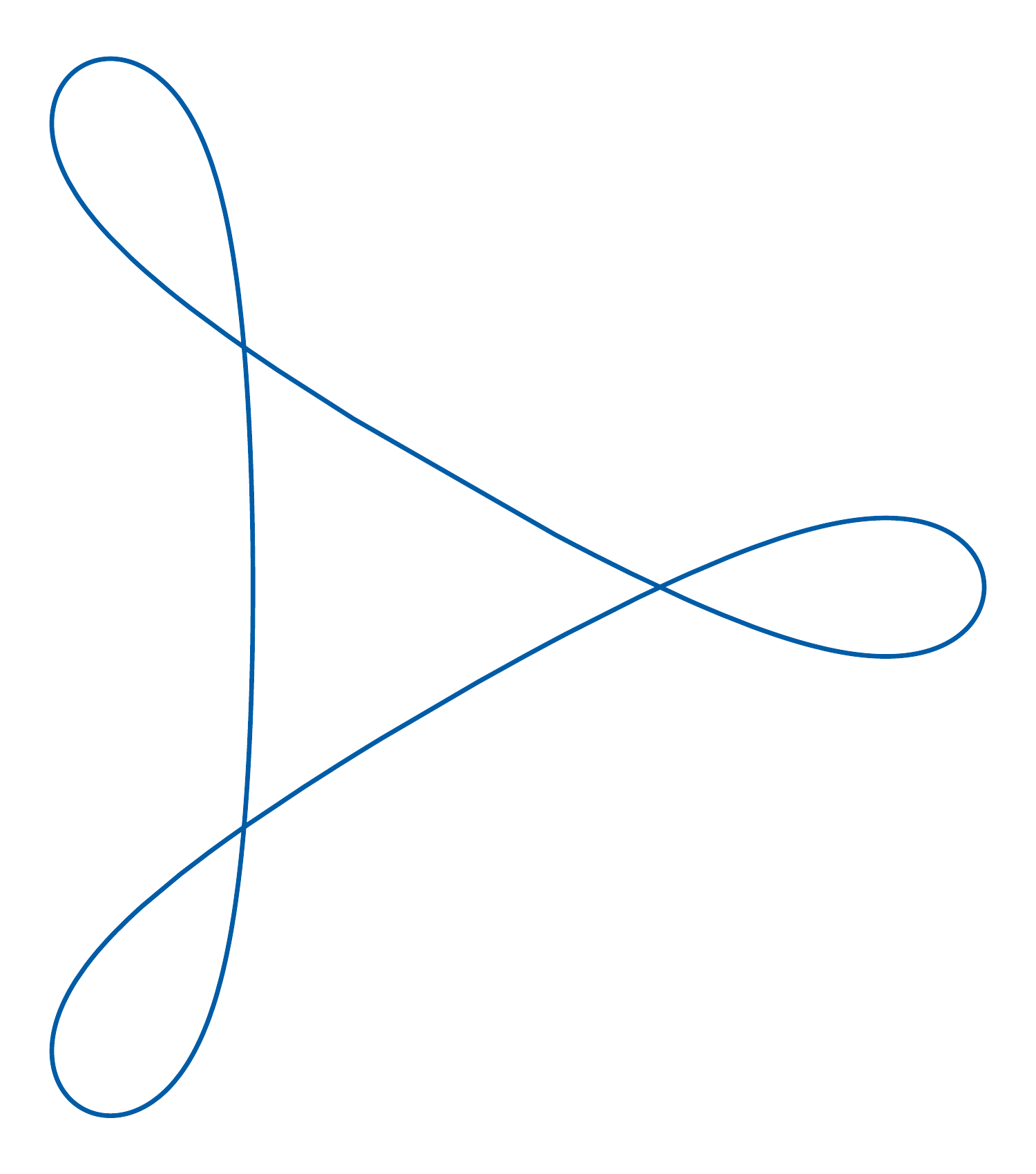}
			\caption*{$\sfrac{m}{n}=\sfrac{2}{3}$}
		\end{subfigure}
		\hfill
		\begin{subfigure}{0.24\textwidth}
			\includegraphics[width=\textwidth]{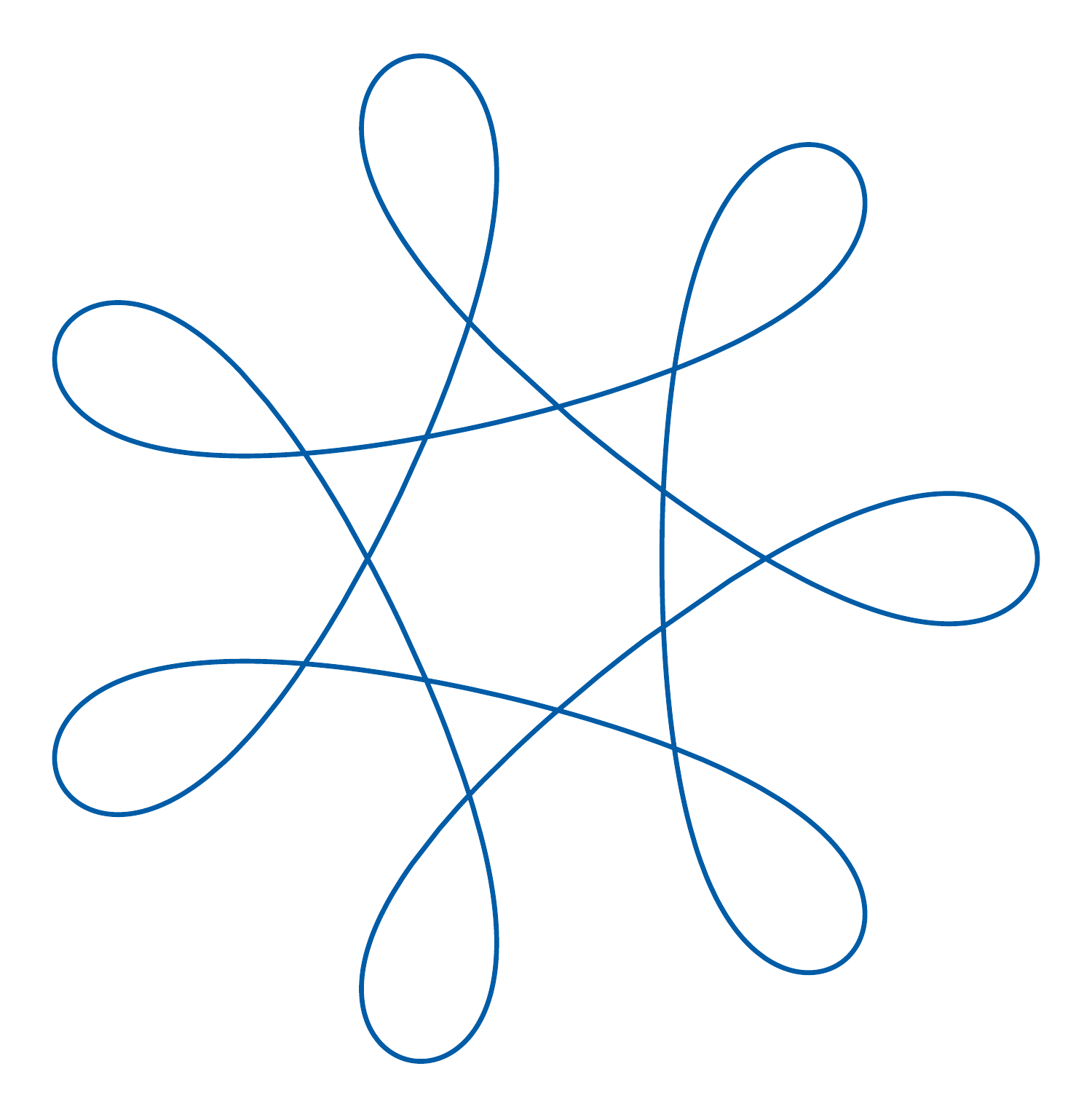}
			\caption*{$\sfrac{m}{n}=\sfrac{5}{7}$}
		\end{subfigure}
		\hfill
		\begin{subfigure}{0.24\textwidth}
			\includegraphics[width=\textwidth]{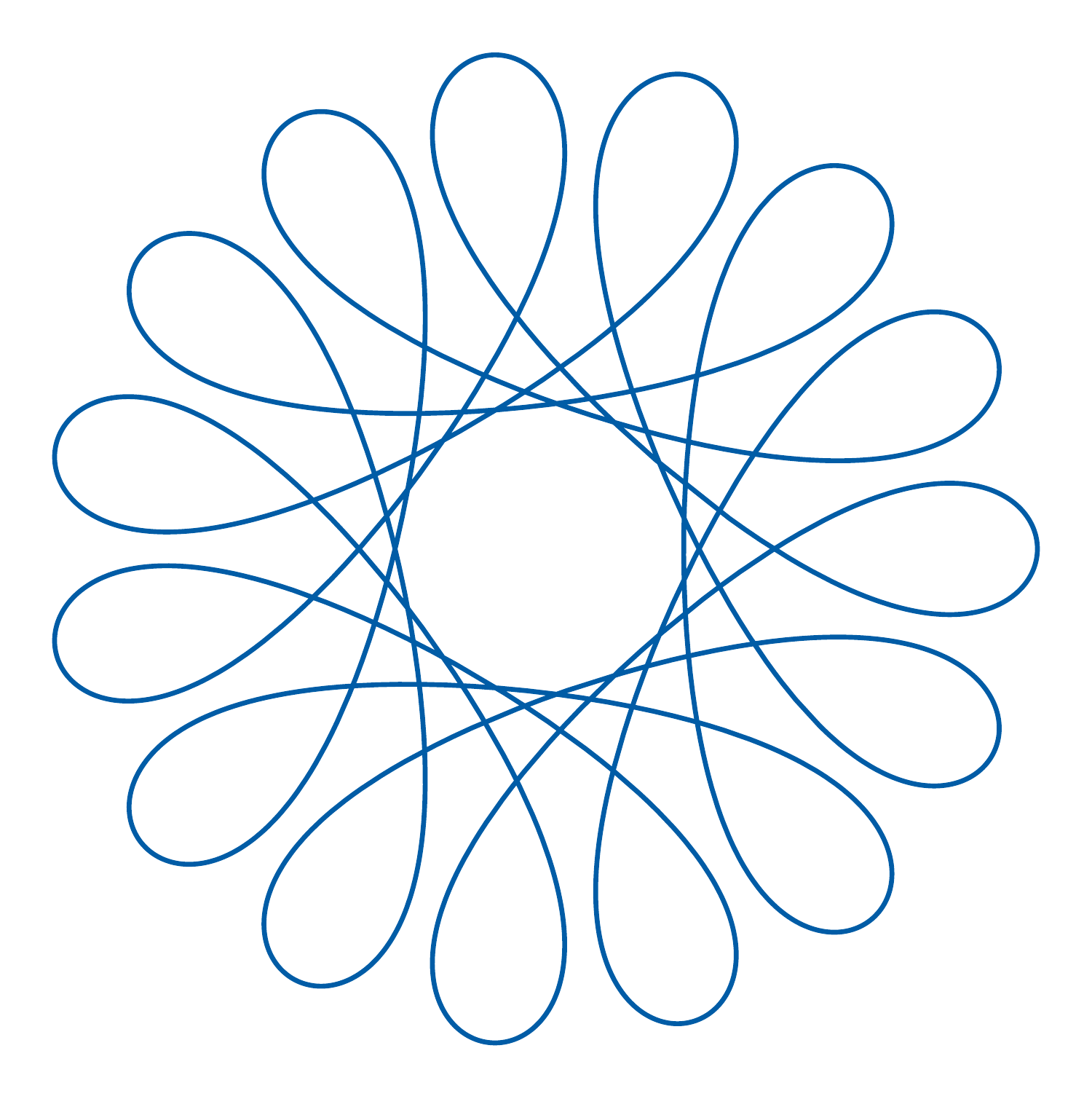}
			\caption*{$\sfrac{m}{n}=\sfrac{11}{15}$}
		\end{subfigure}
		\begin{subfigure}{0.24\textwidth}
			\includegraphics[width=\textwidth]{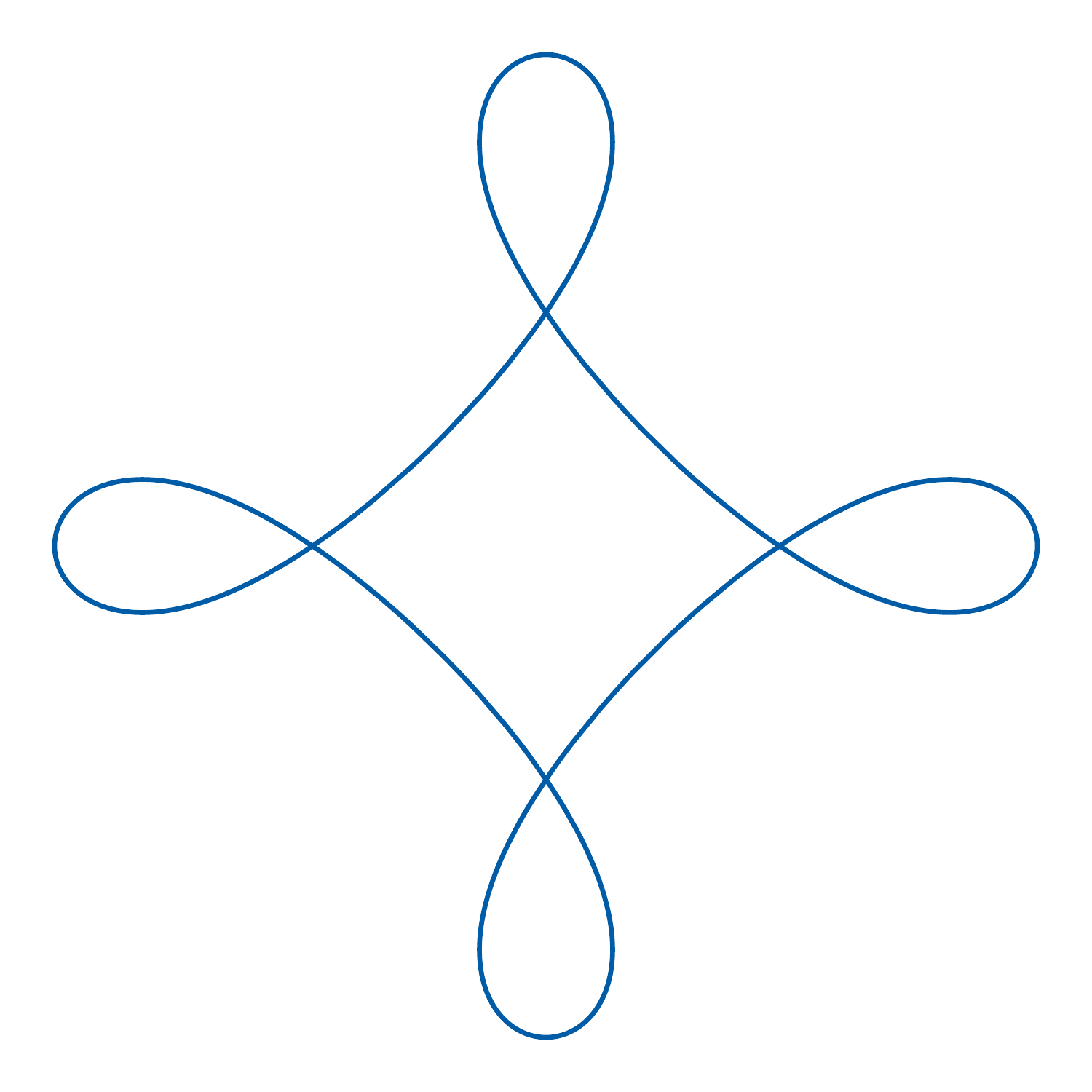}
			\caption*{$\sfrac{m}{n}=\sfrac{3}{4}$}
		\end{subfigure}
		\hfill
		\begin{subfigure}{0.24\textwidth}
			\includegraphics[width=\textwidth]{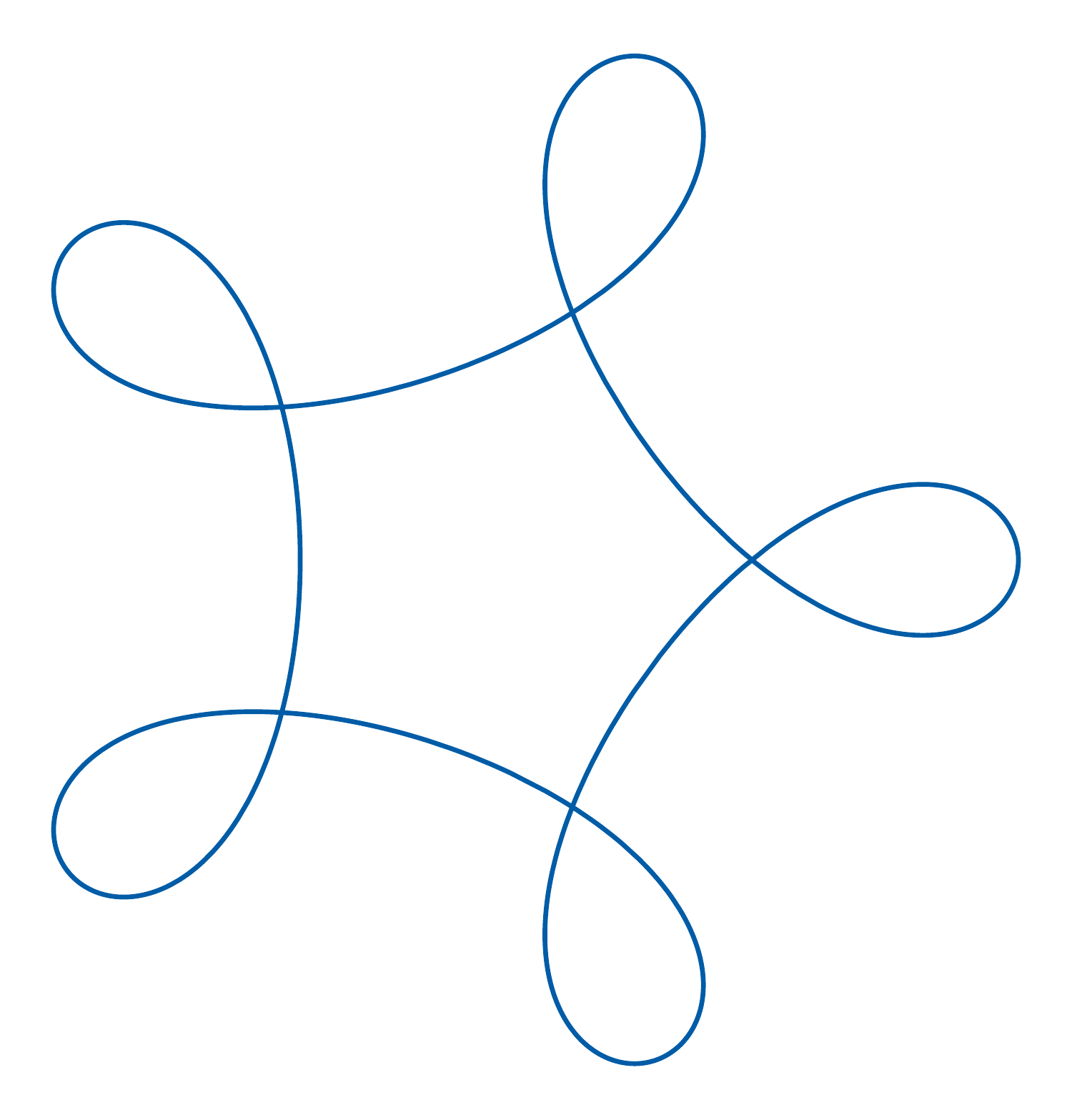}
			\caption*{$\sfrac{m}{n}=\sfrac{4}{5}$}
		\end{subfigure}
		\hfill
		\begin{subfigure}{0.24\textwidth}
			\includegraphics[width=\textwidth]{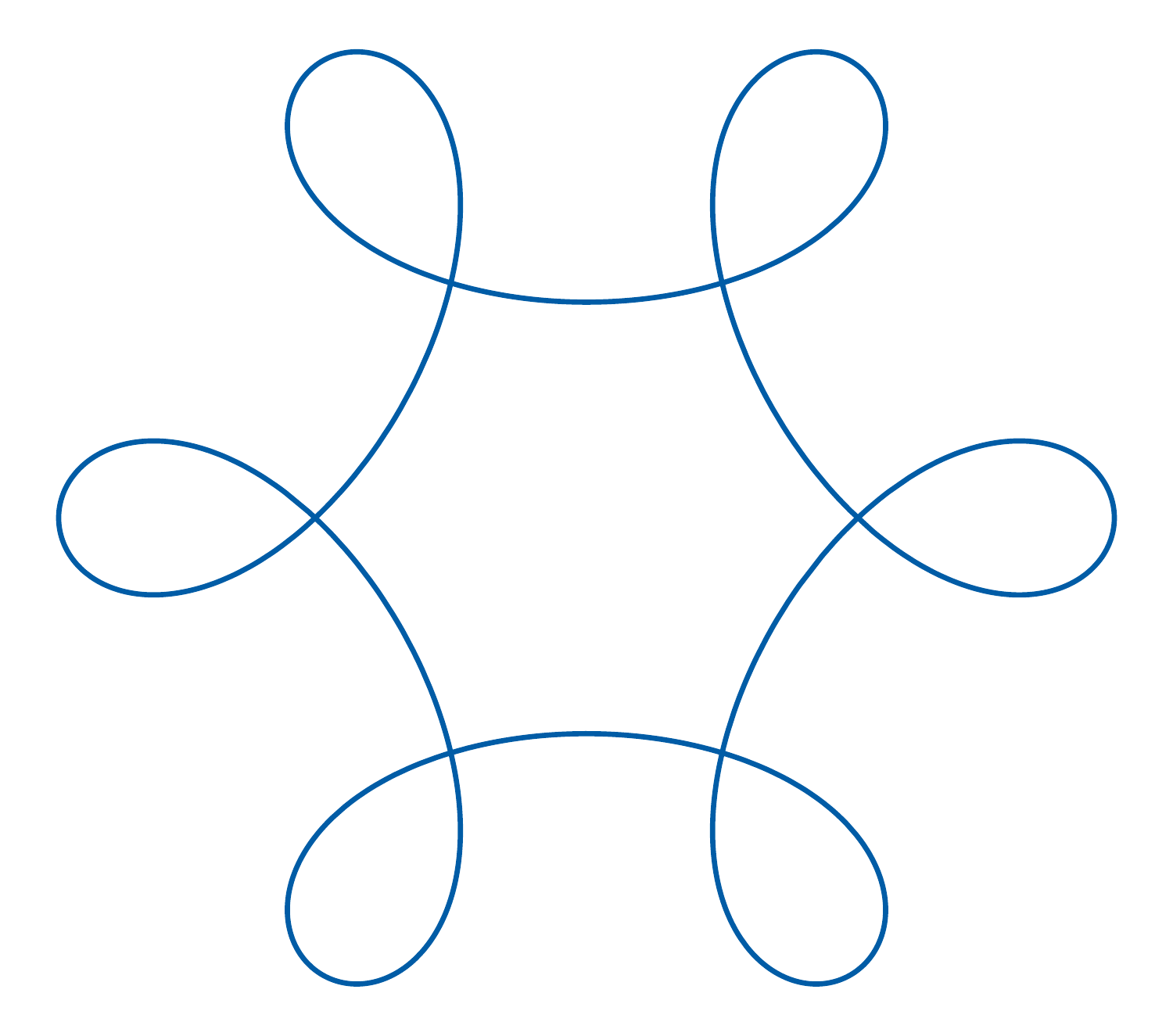}
			\caption*{$\sfrac{m}{n}=\sfrac{5}{6}$}
		\end{subfigure}
		\hfill
		\begin{subfigure}{0.24\textwidth}
			\includegraphics[width=\textwidth]{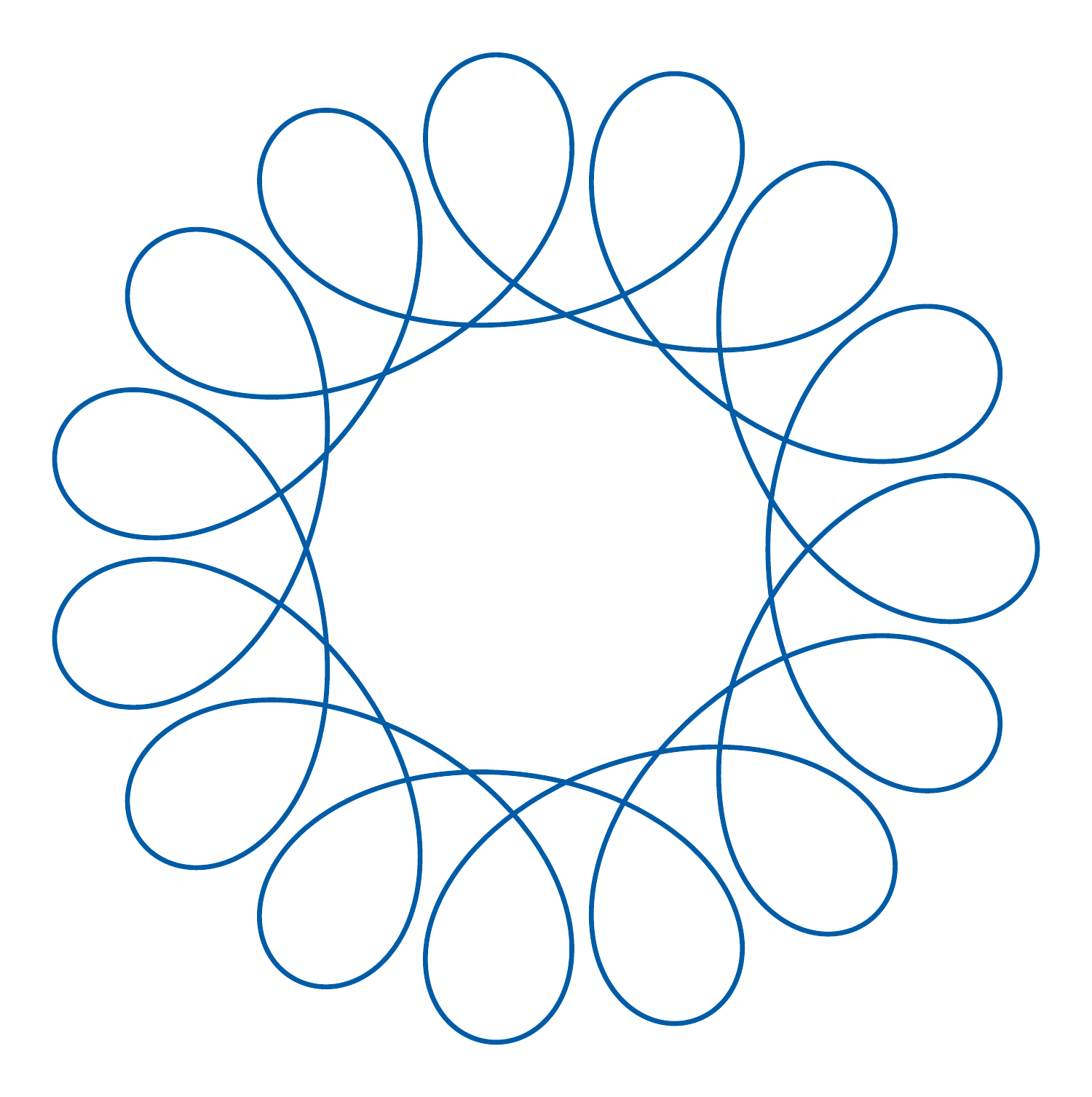}
			\caption*{$\sfrac{m}{n}=\sfrac{13}{15}$}
		\end{subfigure}
		\caption{Shrinkers for the \eqref{APCSF} with several choices of $\sfrac{m}{n}\in(\sfrac{1}{2},1)$ in ascending order from top left to bottom right}\label{abresch_langer_pics}
	\end{figure}
	% SUMMARY OF METHODS
	In order to prove  \thref{main-thm}, we first consider the much more general class of $\lambda$-curves, see Section \ref{sec1}. Using \ODE{} methods for their curvature functions,
	we characterize when such curves close up (\thref{closeuplambda}). We then show that the
	associated period map is strictly decreasing (Section \ref{sec_mon}) and determine its
	range (Section \ref{sec2}). In the remaining sections, we address the two main results stated above. First, we prove \thref{main-thm} by employing geometric estimates, showing the existence of \eqref{APCSF}-shrinkers as a special case of $\lambda$-curves. For \thref{main-thmB}, we apply a result by \cite{WanKon14} related to so-called \textit{Abresch--Langer type} curves.
	\section{Preliminaries}\label{part1}
	Let $x:\bbS^1\to\bbR^2$ be a non-circular immersion that generates a homothetic evolution $X:\bbS^1\times[0,T_\ast)\to\bbR^2$ by \eqref{APCSF} up to a maximal time of existence $T_\ast\in(0,\infty]$. Denote by $\kappa$ and $\sfn$ the curvature and normal of $x$. Finally, let $\bar{\kappa}[x]$ be the average curvature of $x$. Then, by applying a standard separation of variables argument (see e.g.~\cite[§1.4]{Man11}), we find that $x$ must satisfy a \textit{temporal} and a \textit{spatial} \ODE 
	\begin{equation}\label{tempspac}
		\begin{cases}
			\dot{\psi}\psi=-\delta&\text{on $[0,T_\ast)$},\\
			\kappa-\bar{\kappa}[x]=\delta\langle x-x_\ast,\mathsf{n}\rangle&\text{on $\bbS^1$},
		\end{cases}
	\end{equation}
	for some constants $x_\ast\in\bbR^2$ and $\delta>0$. Solving the temporal {\ODE} subject to the initial condition $\psi(0)=1$ implicit in \thref{def_hom} yields $\psi(t)=\sqrt{1-2\delta t}$, which exists up to $T_\ast=1/2\delta$. We therefore interpret $\delta$ as a ``rate of shrinkage'' for the homothetic evolution and find
	\begin{equation}
		\psi(t)=\sqrt{1-\frac{t}{T_\ast}},\quad t\in[0,T_\ast).\label{scaling}
	\end{equation}
	\begin{rem}[singularity formation for homothetic evolutions]
		We emphasize that \eqref{scaling} in particular shows that any non-circular homothetic evolution by \eqref{APCSF} is intrinsically \textit{singular} in the sense that $T_\ast<\infty$. Moreover, due to the scaling properties of the curvature,
		\begin{equation}
			\max_{M^1}\abs{k(\cdot,t)}=\max_{\bbS^1}\abs{\frac{k(\cdot,0)}{\psi(t)}}=\frac{1}{\sqrt{(T_\ast-t)}}\sqrt{T_\ast}\max_{\bbS^1}\abs{\kappa},\quad t\in[0,T_\ast),
		\end{equation}
		so \eqref{APCSF}-shrinkers constitute easy examples of \textit{type-I singular} evolutions.
	\end{rem}
	Geometrically, the evolving family of curves $X(\bbS^1,t)$ generated by $x(\bbS^1)$ ``collapses'' into $x_\ast$ as $t\nearrow T_\ast$. This justifies referring to $x$ (or $x(\bbS^1)$) as an \textit{\eqref{APCSF}-shrinker}. The spatial equation in \eqref{tempspac} can be used to characterize such shrinkers up to a choice for the constants $\delta$ and $x_\ast$.
	\begin{theo}[characterization of \eqref{APCSF}-shrinkers]\thlabel{characterize_shrinerks}
		Up to translation, uniform scaling and reparametrization, $x:\bbS^1\to\bbR^2$ is a non-circular \eqref{APCSF}-shrinker if and only if
		\begin{equation}
			\kappa=\inner{x}{\mathsf{n}}+\bar{\kappa}[x]\quad \text{on $\bbS^1$}.\label{soleq}
		\end{equation}
	\end{theo}
	\begin{proof}
		If $x$ is an \eqref{APCSF}-shrinker, then by the above discussion, the spatial equation in \eqref{tempspac} is necessarily satisfied. Applying an appropriate translation and scaling to $x$ yields $\delta=1$ and $x_\ast=0$. Assume conversely that \eqref{soleq} holds. Then the homothetic family $X(\cdot,t)\coloneqq \psi(t) x$ with $t\in[0,\frac{1}{2})$ satisfies \eqref{APCSF} with initial datum $X(\cdot,0)=x$, possibly after reparametrization.
	\end{proof}
	Integrating \eqref{soleq} over $\bbS^1$ and canceling redundant terms provides an unsurprising but welcome reassurance of \thref{zeroarea}.
	\begin{coro}[area of \eqref{APCSF}-shrinkers, again]\thlabel{zeroarea2}
		If $x:\bbS^1\to\bbR^2$ is a non-circular \eqref{APCSF}-shrinker, then $\vol[x]=0$.
	\end{coro}
	\begin{rem}[the Abresch--Langer case]\thlabel{ALcse0}
		The 2$^{\text{nd}}$-order integro-differential equation \eqref{soleq} is an example of a \textit{soliton equation}. The analogous equation for \eqref{CSF} or more generally \eqref{MCF} reads $\mathrm{H}=\inner{x}{\mathsf{n}}$ on $M^d$. Proving \thref{abresch-langer} involves linking the solutions of the \eqref{CSF}-soliton equation to the geometric properties of the corresponding curves.
	\end{rem}
	
	\subsection{$\lambda$-curves}\label{sec1}
	A general difficulty that we are faced with regarding \eqref{soleq} is the non-local term $\bar{\kappa}$. To mitigate this, we instead consider a more general notion.
	\begin{defi}[$\lambda$-surfaces and -curves]
		Let $\lambda\in\bbR$ and $x:M^d\to\bbR^{d+1}$ be an immersion. If there are $\delta>0$ and an $x_\ast\in \bbR^{d+1}$ such that
		\begin{equation}\label{lambdasurfaceequation}
			\rmH=\delta\inner{x-x_\ast}{\sfn}+\lambda,
		\end{equation}
		we call $x$ (or $x(M^d)$) a \emph{$\lambda$-surface}. In the case $d=1$, we refer to $x$ (or $x\qty(M^1)$) as a \emph{$\lambda$-curve}.
	\end{defi}
	Such objects were originally introduced by \cite{McGRos15} in the study of certain variational problems pertaining to weighted area functionals. Independently of that, the term ``$\lambda$-surface'' was coined by \cite{CheWei18}. Taking $\lambda=0$ recovers the \eqref{MCF}-soliton equation stated in \thref{ALcse0}. 
	Therefore, $0$-surfaces and  closed non-circular $0$-curves precisely coincide with \eqref{MCF}-shrinkers and Abresch--Langer curves, respectively. Taking $\rmH$ to be constant on $M^d$, one also trivially finds that planes, spheres and cylinders with appropriate distances from the origin and radii constitute $\lambda$-surfaces for any $\lambda\in\bbR$.
	Excluding those, there are several non-CMC $\lambda$-surfaces with $\lambda\neq 0$.
	\begin{ex}[non-trivial $\lambda$-surfaces]
		Any of the following are examples of $\lambda$-surfaces in their respective ambient spaces.
		\begin{enumerate}
			\item The family of symmetric closed $\lambda$-curves $\gamma_\lambda\subseteq\bbR^2$ found by \cite{Cha17} and all related cylinders $\gamma_\lambda\times \bbR^{d-1}\subseteq\bbR^{d+1}$.
			\item The toroidal embeddings and non-embedded immersions given by \cite{CheWei21}.
			\item Embedded \cite{CheLaiWei24} and non-embedded \cite{LiWei23} sphere-immersions.
			\item The $(\operatorname{SO}(d)\times\operatorname{SO}(d))$-invariant odd-dimensional embedded hypersurface diffeomorphic to $\bbS^{d-1}\times\bbS^{d-1}\times\bbS^1$ discovered by \cite{Ros19}.
		\end{enumerate}
	\end{ex}
	In light of \thref{characterize_shrinerks}, finding \eqref{APCSF}-shrinkers clearly amounts to the study of \eqref{lambdasurfaceequation} with $d=1$ and an appropriate choice of $\lambda\geq0$. Let us collect some observations for the curve case.
	\begin{lemma_no_break}[elementary properties of $\lambda$-curves]
		\leavevmode
		\begin{enumerate}
			\item Let $x:M^1\to\bbR^2$ be an immersion. Then, up to translations and uniform scaling, $x(M^1)$ is a $\lambda$-curve if and only if\label{no_break_1}
			\begin{equation}\label{lamdacurveequation}
				\kappa=\inner{x}{\mathsf{n}}+\lambda\qcomma{\text{on }M^1}.
			\end{equation}
			\item Let $\gamma\subseteq \bbR^2$ be a $\lambda$-curve. Then $\gamma$ is either a straight line or strictly locally convex.\label{no_break_2}
		\end{enumerate}
	\end{lemma_no_break}
	\begin{proof} 
		Applying an appropriate translation and scaling of $x$ in \eqref{lambdasurfaceequation} yields \eqref{lamdacurveequation}, proving \ref{no_break_1}.
		For \ref{no_break_2}, we let $x:I\to\bbR^2$ be a unit-speed  parametrization of $\gamma$ for some interval $I\subseteq\bbR$ and $\partial_s(\cdot)$ be the corresponding arc-length derivative. Then \eqref{lamdacurveequation} along with the Frenet equations imply
		\begin{equation*}
			\partial_s \kappa = \inner{\partial_s x}{\sfn}+\inner{x}{\partial_s \sfn}=-\kappa\inner{x}{\sft}\quad \text{in $I$.}
		\end{equation*}
		Assume there is an $s_0\in I$ such that $\kappa(s_0)=0$. Then, by \ODE-uniqueness, it follows that $\kappa\equiv 0$. Thus, $\kappa$ is either identically zero or does not change its sign.
	\end{proof}
	Because non-trivial $\lambda$-curves have strictly positive curvature up to orientation, we may choose a particularly helpful parametrization. Suppose that $x:M^1\to\bbR^2$ is a positively oriented parametrization of a non-trivial $\lambda$-curve, that is, the curvature function is strictly positive on $M^1$. We consider a lift of the corresponding unit tangent field $\sft:M^1\to\bbS^1$, i.e., a map $\theta:M^1\to\bbR$ satisfying
	\begin{equation}
		\sft(p)=(\cos \theta(p),\sin \theta(p))\qcomma p\in M^1.
	\end{equation}
	Then $\theta$ is unique up to addition of $2\pi$ and constitutes an orientation-preserving diffeomorphism onto its image $\theta(M^1)$. Geometrically, $\theta(p)$ is the tangent angle, i.e., the angle enclosed by the tangent line at $p\in M^1$ (or $x(p)$) and the $\mathrm{e}_1$-axis counted with multiplicity. We will refer to the  reparametrization $x\circ \theta^{-1}:\theta(M^1)\to\bbR^2$ as \textbf{tangential polar coordinates} for $x(M^1)$.
	\begin{rem}[tangent angle derivative and the support function]\thlabel{refback}
		Let $I\subseteq\bbR$ be an interval and $x:I\to \bbR^2$ a positively oriented unit-speed parametric curve with curvature $\kappa:I\to(0,\infty)$ and $\theta:I\to\theta(I)$ the corresponding tangent angle map. 
		\begin{enumerate}
			\item The Frenet equations show that $\partial_s\theta=\kappa$ on $I$. This justifies denoting the \textbf{tangent angle derivative} as
			\begin{equation}
				\qty(\cdot)^\prime\coloneqq\pdv{\theta}\coloneqq \frac{1}{\kappa}\pdv{s}.\label{tangentanglederivative}
			\end{equation}
			\item  By a slight abuse of notation, we write $x\coloneqq x\circ \theta^{-1}$, $\sfn\coloneqq \sfn\circ\theta^{-1}$ and $\kappa\coloneqq \kappa\circ \theta^{-1}$ whenever we are considering a parametric curve in tangential polar coordinates. This convention along with \eqref{tangentanglederivative} provides an alternative formula for the length of $x$,
			\begin{equation}
				\len[x]=\int_I\dd{s}=\int_{\theta(I)}\frac{\dd{\theta}}{\kappa(\theta)}.\label{lenform}
			\end{equation}
			\item Finally, we introduce the \textbf{support function} $\sigma:\theta(I)\to\bbR$ given by
			\begin{equation}
				\sigma(\theta)\coloneqq \inner{x(\theta)}{\sfn(\theta)}\qcomma{\theta\in \theta(I)}.
			\end{equation}\label{comon}
			Geometrically, $\sigma(\theta)$ corresponds to the \textit{signed} distance from the origin to the tangent line at $x(\theta)$ for every $\theta\in\theta(I)$. Again, using the Frenet formulas, one can show that
			\begin{equation}
				\frac{1}{\kappa}=\sigma^{\prime\prime}+\sigma\qcomma{\text{on }\theta(I)}.\label{rel_kappa_support}
			\end{equation}
		\end{enumerate}
	\end{rem}
	We finish this section by describing $\lambda$-curves in terms of an {\ODE} of their curvatures.
	\begin{theo}[characterization of $\lambda$-curves]\thlabel{lambda_charac}
		Suppose $\lambda\in\bbR$ and $x:\bbR\to\bbR^2$ is an immersion given in tangential polar coordinates. Then, up to uniform scaling and inversion of orientation,  $x$ parametrizes a $\lambda$-curve if and only if the curvature function $\kappa:\bbR\to(0,\infty)$ of $x$ is a solution of
		\begin{equation}
			\kappa^{\prime\prime}=\frac{1}{\kappa}-\kappa+\lambda\qcomma{\text{on }\bbR}
			\tag*{$\qty(\textsc{ode})_{\lambda}$}\labeltarget{ODE}
		\end{equation}
	\end{theo}
	\begin{proof}
		The result follows from \eqref{lamdacurveequation} along with \eqref{rel_kappa_support}.
	\end{proof}
	\subsection{Closing $\lambda$-curves}
	As we are looking for \eqref{APCSF}-shrinkers, we are particularly interested in \textit{closed} $\lambda$-curves with $\lambda\geq 0$. By virtue of \thref{lambda_charac}, we may equivalently determine all solutions $\kappa$ of $\qty(\hyperlink{ODE}{\textsc{ode}})_{\lambda}$ that correspond to parametric curves given in tangential polar coordinates $x:\bbR\to\bbR^2$ that ``close up'' eventually.
	In this section we formulate a characterization for the initial values imposed on $\qty(\hyperlink{ODE}{\textsc{ode}})_{\lambda}$ so that $x$ indeed constitutes a closed curve.
	With a view to establishing such periodic behavior, we interpret $\theta\mapsto\kappa(\theta)$ as a mechanical motion in a Hamiltonian system. To that end we notice that the \textbf{total energy} of $\kappa$
	\begin{equation}
		E=\frac{1}{2}\qty(\kappa^{\prime})^2+\frac{\kappa^2}{2}-\lambda\kappa-\log \kappa=\text{const.}\quad\parbox{2.5cm}{\centering\text{for all solutions}\\\text{$\kappa$ of $\qty(\hyperlink{ODE}{\textsc{ode}})_{\lambda}$},}\tag*{$\qty(\textsc{fi})_{\lambda,E}$}\labeltarget{FI}
	\end{equation}
	serves as a well-known first integral for $\qty(\hyperlink{ODE}{\textsc{ode}})_{\lambda}$. The term corresponding to the potential energy of  $\kappa$ is instrumental for the remaining discussion.
	\begin{figure}[htbp]
		\centering
		\begin{tikzpicture}[
			blue dot/.style={circle, fill=blue, inner sep=1.5pt},
			red dot/.style={circle, fill=red, inner sep=1.5pt}
			]
			\def\E{2.0}           % Das Energie-Niveau E
			\def\kmin{0.126}      % k^-,
			\def\kmax{3.77}       % k^+,
			\def\kzero{1.618}     % k^0,
			\def\Vzero{-0.792}    % V^0,
			
			\begin{axis}[
				% --- BREITE ANPASSEN ---
				width = 12cm, 
				height = 8cm,
				% -----------------------
				axis lines = middle,
				axis line style = {thick, ->, >=stealth},
				xlabel = {$\kappa$},
				ylabel = {$V_\lambda(\kappa)$},
				xlabel style = {anchor=north west},
				ylabel style = {anchor=south east},
				xmin = -0.5, xmax = 5.5,
				ymin = -1.5, ymax = 4.5,
				xtick = \empty,
				ytick = \empty,
				enlargelimits = false,
				clip = false
				]
				
				\addplot [
				very thick, 
				black, 
				samples=150, 
				smooth, 
				domain=0.05:4.75
				] {0.5*x^2 - x - ln(x)};
				
				% Energie-Linie
				\draw [dashed] (axis cs: 0, \E) -- (axis cs: \kmax+0.8, \E) 
				node [right] {$E$};
				
				% Linker Schnittpunkt
				\draw [dashed, blue] (axis cs: \kmin, \E) -- (axis cs: \kmin, 0)
				node [right=8pt, below=0pt, text=black, font=\small] {$\kappa_\lambda^-(E)$};
				\node [blue dot] at (axis cs: \kmin, \E) {};
				
				% Rechter Schnittpunkt
				\draw [dashed, blue] (axis cs: \kmax, \E) -- (axis cs: \kmax, 0)
				node [below=0pt, text=black, font=\small] {$\kappa_\lambda^+(E)$};
				\node [blue dot] at (axis cs: \kmax, \E) {};
				
				% Minimum
				\draw [dashed, red, thick] (axis cs: \kzero, \Vzero) -- (axis cs: \kzero, 0)
				node [above=0pt, text=black, font=\small] {$\kappa_\lambda^0$};
				\draw [dashed, red, thick] (axis cs: \kzero, \Vzero) -- (axis cs: 0, \Vzero)
				node [left=2pt, text=black, font=\small] {$V_\lambda^0$};
				\node [red dot] at (axis cs: \kzero, \Vzero) {};
				
			\end{axis}
		\end{tikzpicture}
		\caption{The $\lambda$-potential for some arbitrary choice of $\lambda\geq0$ and its related quantities.}\label{lambda_pot_pic}
	\end{figure}
	\begin{defi}[$\lambda$-potential and related quantities]\thlabel{def_V}
		Suppose $\lambda\geq 0$. 
		\begin{enumerate}
			\item We call the strictly convex function
			\begin{equation}
				V_\lambda(\kappa)\coloneqq \frac{\kappa^2}{2}-\lambda\kappa-\log \kappa,\quad \kappa\in(0,\infty),
			\end{equation}
			the \textbf{$\lambda$-potential}.
			\item We denote the minimal value of $V_\lambda$ as\label{def_k_0}
			\begin{equation}
				V_\lambda^0\coloneqq\min_{\kappa>0}V_\lambda(\kappa)\quad\text{uniquely attained at}\quad \kappa^0_\lambda\coloneqq\frac{\lambda}{2}+\sqrt{\frac{\lambda^2}{4}+1}.
			\end{equation}
			\item For $E>V^0_\lambda$ we let $\kappa_\lambda^\pm(E)$ be the two distinct solutions of $V_\lambda(\kappa^\pm_\lambda(E))=E$ with $\kappa_\lambda^-(E)<\kappa_\lambda^+(E)$.
		\end{enumerate}
	\end{defi}
	Any given $\lambda$-potential decreases on $\qty(0,\kappa_\lambda^0)$ and increases on $\qty(\kappa_\lambda^0,\infty)$. The location of its minimal value is characterized by
	\begin{equation}
		\eval{\dv{V_\lambda}{\kappa}}_{\kappa=\kappa_\lambda^0}=0,\quad\text{i.e.,}\quad \kappa_\lambda^0-\lambda=\frac{1}{\kappa_\lambda^0}.\label{min_ident}
	\end{equation}
	Moreover, we note that the overall ``shape'' of $V_\lambda$ is skewed towards the positive direction (see e.g.~Figure \ref{lambda_pot_pic}). The following makes this rigorous.
	\begin{lemma}[the $\lambda$-potential is skewed to the right]\thlabel{skewed}
		Let $\lambda\geq 0$ and $E>V_\lambda^0$. Then $\kappa_\lambda^0-\kappa_\lambda^-(E)<\kappa_\lambda^+(E)-\kappa_\lambda^0$.
	\end{lemma}
	\begin{proof}
		Put $\varphi(h)\coloneqq V_\lambda\qty(\kappa_\lambda^0+h)-V_\lambda\qty(\kappa_\lambda^0-h)$ for $h\in\qty(0,\kappa_\lambda^0)$. Using \eqref{min_ident}, we calculate
		\begin{equation}
			\dv{\varphi}{h}=2\qty(\kappa_\lambda^0-\lambda)-\frac{1}{\kappa_\lambda^0+h}-\frac{1}{\kappa_\lambda^0-h}=-\frac{2h^2}{\kappa_\lambda^0\qty(\qty(\kappa_\lambda^0)^2-h^2)}<0
		\end{equation}
		Additionally observing that $\varphi(h)\to 0$ as $h\searrow 0$, we infer that $\varphi<0$ on $\qty(0,\kappa_\lambda^0)$. Thus $V_\lambda\qty(\kappa_\lambda^0+h)<V_\lambda(\kappa_\lambda^0-h)$ and substituting $h=\kappa_\lambda^0-\kappa_\lambda^-(E)$ gives
		\begin{equation}
			V_\lambda\qty(2\kappa_\lambda^0-\kappa_\lambda^-(E))<V_\lambda\qty(\kappa_\lambda^-(E))=V_\lambda\qty(\kappa_\lambda^+(E)).\label{weinfer}
		\end{equation}
		Note that $2\kappa_\lambda^0-\kappa_\lambda^-(E),\kappa_\lambda^+(E)\in(\kappa_\lambda^0,\infty)$. Since $V_\lambda$ is increasing on $\qty(\kappa_\lambda^0,\infty)$ we infer $2\kappa_\lambda^0-\kappa_\lambda^-(E)<\kappa_\lambda^+(E)$ from \eqref{weinfer}.
	\end{proof}
	The conservative system $\kappa^{\prime\prime}=-\nabla V_\lambda(\kappa)$ is clearly equivalent to $\qty(\hyperlink{ODE}{\textsc{ode}})_{\lambda}$ for any $\lambda\geq 0$. Furthermore, as $V_\lambda$ satisfies $V_\lambda(\kappa)\to \infty$ for both $\kappa\to 0$ and $\kappa\to \infty$, the sublevel set $\qty{V_{\lambda}\leq E}$ is compact for any $E>V_{\lambda}^0$. This implies that trajectories $\theta\mapsto\kappa(\theta)$ are periodic. Thus any $\lambda$-curve must have a periodic curvature function when parametrized by tangential polar coordinates. In particular, $\kappa$ attains $\kappa_\lambda^\pm (E)$ as its extremal values.
	\begin{nota}[$\qty(\hyperlink{ODE}{\textsc{ode}})_{\lambda}$ as an initial value problem]\thlabel{nota1}
		Let $\lambda\geq 0$ and $E>V_\lambda^0$ be an \textbf{initial energy}. We denote $\kappa_{\lambda,E}$ as the unique solution of the initial value problem 
		\hypertarget{IVP}{\begin{equation}
				\begin{dcases}
					\kappa^{\prime\prime}=\frac{1}{\kappa}-\kappa+\lambda&\text{on $\bbR$}\\
					\kappa(0)=\kappa^-_\lambda(E)\\
					\kappa^\prime(0)=0.
				\end{dcases}\tag*{$\qty(\textsc{ivp})_{\lambda,E}$}
		\end{equation}}
		Additionally, we denote any member of the family of congruent parametric curves given in tangential polar coordinates with curvature function $\kappa_{\lambda,E}$ as $x_{\lambda,E}$. Finally, we put $\gamma_{\lambda,E}\coloneqq x_{\lambda,E}(\bbR)$.
	\end{nota}
	\begin{rem}[correspondence between $\lambda$-curves and $\qty(\hyperlink{ODE}{\textsc{ode}})_{\lambda}$]
		Suppose $\kappa$ is an arbitrary solution of $\qty(\hyperlink{ODE}{\textsc{ode}})_{\lambda}$. Then a rotation of the corresponding curve by an appropriate angle results in $\kappa$ assuming its minimum at $\theta=0$. Thus, up to a similarity, any non-trivial $\lambda$-curve may be realized as $\gamma_{\lambda,E}$ for some initial energy $E>V^0_\lambda$.
	\end{rem}
	Now that we have firmly established that $\kappa_{\lambda,E}$ is periodic for any $\lambda\geq 0$ and $E>V^0_\lambda$, the next natural question concerns its (minimal) period.
	\begin{defi}[energy-(semi)-period map]
		For $\lambda\geq0$ and $E>V^0_\lambda$ we denote the \emph{semi-period of $\kappa_{\lambda,E}$} as
		\begin{equation}
			\Theta_\lambda(E)\coloneqq \int_{\kappa^-_\lambda(E)}^{\kappa^+_\lambda(E)}\frac{\dd{\kappa}}{\sqrt{2\qty(E-V_\lambda(\kappa))}}.\label{semiperiod}
		\end{equation}
	\end{defi}
	\begin{rem}[interpretations of $\Theta_\lambda(E)$]\thlabel{phys_travel_time}
		Suppose that $\lambda\geq 0$ and\linebreak $E>V_\lambda^0$.
		\begin{enumerate}
			\item Geometrically, $\Theta_{\lambda}(E)$ is the turning angle of the tangent line between consecutive vertices $\theta^\pm\in\bbR$ of $x_{\lambda,E}$. Taking $\theta^-=0$ and $\theta^+>0$ as the least positive angle with $\kappa_{\lambda,E}(\theta^+)=\kappa^+_\lambda(E)$ gives
			\begin{equation}
				\theta^+-\theta^-=\int_0^{\theta^+}\dd{\theta}=\int_{\theta^{-1}(0)}^{\theta^{-1}(\theta^+)}\kappa(s)\dd{s}=\Theta_\lambda(E)
			\end{equation} 
			because of $\qty(\hyperlink{FI}{\textsc{fi}})_{\lambda,E}$. Also see \cite[Lemma 2.1]{Cha17}.
			\item \label{phys_travel_time_2} Physically, $\Theta_{\lambda}(E)$ corresponds to the (minimal) ``travel time'' from $\kappa_\lambda^\pm(E)$ to $\kappa_\lambda^\mp(E)$ within the potential $V_\lambda$. More generally, if we consider a trajectory $\tau\mapsto\xi(\tau)$ in a potential $U$ with total energy $E$, then
			\begin{equation}
				\tau_2-\tau_1=\int_{\xi(\tau_1)}^{\xi(\tau_2)}\frac{\dd{\xi}}{\sqrt{2(E-U(\xi))}}.
			\end{equation}
		\end{enumerate}
	\end{rem}
	Given any closed curve, the corresponding curvature function is always periodic. The converse, however, is not true. We therefore need to determine initial energies $E>V^0_\lambda$ such that $x_{\lambda,E}$ parametrizes a closed curve. The next two results resolve this while also providing information on the geometry of~$\gamma_{\lambda,E}$.
	\begin{lemma}[closing curve criterion; {\cite{ArrGarMen08}}]
		Suppose that $\kappa:\bbR\to\bbR$ is periodic with minimal period $P>0$ and let $x:\bbR\to\bbR^2$ be an immersion with curvature $\kappa$. Then, $x$ closes up in $[0,nP]$ with $n\in\bbN$, $n\geq 2$, if and only if there is an $m\in\bbZ$ such that
		\begin{equation}
			\frac{1}{2\pi}\int_0^P\kappa(s)\dd{s}=\frac{m}{n}\in\bbQ\setminus\bbZ.
		\end{equation}
	\end{lemma}
	\begin{coro}[criterion for and geometry of closed $\lambda$-curves]\thlabel{closeuplambda}
		Let $\lambda\geq 0$ and $E>V^{0}_\lambda$.
		\begin{enumerate}
			\item If $\Theta_\lambda(E)/\pi\in\bbQ\setminus\bbZ$, then $x_{\lambda,E}$ is closed; if
			$\Theta_\lambda(E)/\pi\notin\bbQ$, it is not.\thlabel{closeuplambda1}
			\item Suppose $x_{\lambda,E}$ is closed with $\Theta_\lambda(E)=\sfrac{m\pi}{n}$ for a pair of coprime $n,m\in\bbN$. Then $x_{\lambda,E}$ has tangent turning index $m$ and $\gamma_{\lambda,E}$ exhibits $n$-fold rotational symmetry. \label{closeuplambda2}
		\end{enumerate}
	\end{coro}
	The above result states that the energy of any $\lambda$-curve determines its \textit{global} geometry. The next and final result of this section shows that the parameter $\lambda\geq 0$ chiefly affects the \textit{local} geometry of $\gamma_{\lambda,E}$.
	\begin{lemma}[extremal curvatures of $x_{\lambda,E}$ in the limit]\thlabel{extrem}
		Let $E\in\bbR$. Then $E>V_\lambda^0$ for all sufficiently large $\lambda$, and the limits
		\begin{equation}
			\lim\limits_{\lambda\to\infty}\qty(\kappa^+_\lambda(E)-\lambda)=\infty\quad\text{and}\quad\lim\limits_{\lambda\to\infty}\kappa^-_\lambda(E)=0\quad\text{hold.}
		\end{equation}
	\end{lemma}
	\begin{proof}
		For the first statement, observe that 
		\begin{equation}
			V_\lambda^0=\min_{\kappa>0}V_\lambda(\kappa)\leq V_\lambda(1)=\frac{1}{2}-\lambda\xrightarrow{\lambda\to\infty}-\infty.
		\end{equation}
		We are only considering the first limit as the second  one follows analogously. Consider therefore $\kappa\geq \kappa^0_\lambda>1$. Then, the estimate
		\begin{equation}\label{quadratic}
			V_\lambda(\kappa)=\frac{1}{2}(\kappa-\lambda)^2-\frac{\lambda^2}{2}-\log\kappa\leq \frac{1}{2}(\kappa-\lambda)^2-\frac{\lambda^2}{2}\eqqcolon \tilde{V}_\lambda(\kappa)
		\end{equation}
		follows. Inverting the above inequality, i.e., solving $\tilde{V}_\lambda(\kappa)=E$, gives
		\begin{equation}
			\kappa_\lambda^+(E)=\qty(V_\lambda|_{[\kappa_\lambda^0,\infty)})^{-1}(E)\geq \tilde{V}_\lambda^{-1}(E)=\lambda+\sqrt{2E+\lambda^2}
		\end{equation}
		and the result follows. For the limit regarding $\lambda\mapsto\kappa^-_\lambda(E)$, one assumes $\kappa\leq\kappa^0_\lambda$ and estimates the quadratic term in \eqref{quadratic}.
	\end{proof}
	\subsection{$\Theta_\lambda(\cdot)$ as a function of amplitude}
	We introduce a convenient choice of coordinates for the energy-period map. Rather than using initial energies $E>V^0_\lambda$ to parametrize the space of solutions of $\qty(\hyperlink{ODE}{\textsc{ode}})_{\lambda}$, we consider the ratio of the extremal values $\kappa^\pm_\lambda(E)$. Given $\lambda\geq 0$ we notice that $E\mapsto\kappa^{\mp}_\lambda(E)$ is monotonically decreasing (respectively increasing) on $(V^0_\lambda,\infty)$. Therefore,
	\begin{equation*}
		\Phi_\lambda:\qty(V_{\lambda}^0,\infty)\longrightarrow(1,\infty),\qquad E\longmapsto\frac{\kappa^+_\lambda(E)}{\kappa^-_\lambda(E)}
	\end{equation*}
	constitutes a diffeomorphism. For $E>V_\lambda^0$, we put $r\coloneqq\Phi_\lambda(E)$. Then, by abuse of notation, we write $\Theta_\lambda(r)\coloneqq \Theta_\lambda(\Phi_\lambda^{-1}(r))$.
	We will see that this choice of coordinates swaps an easier domain of integration for a more involved integrand. As preparation, we also introduce the following auxiliary quantity.
	\begin{nota}[positive root function]\thlabel{defining_property}
		For $\lambda\geq 0$ and $r>1$, we denote the unique positive root of the $2^{\text{nd}}$-degree polynomial
		\begin{equation}
			p_{\lambda,r}[K]\coloneqq K^2\!-\!\frac{2\lambda}{r+1}K\!-\!\frac{2\log r}{r^2-1}\quad\text{as}\quad	\eta_\lambda(r)=\frac{\lambda}{r+1}\!+\!\sqrt{\frac{\lambda^2}{(r+1)^2}\!+\!\frac{2\log r}{r^2-1}}.
		\end{equation}
	\end{nota}
	\begin{rem}[monotonicity of $\Phi_\lambda(E)$ and $\eta_\lambda(r)$ in $\lambda$]\thlabel{prop_of_r_eta}
		Easy calculations show that $\lambda\mapsto\Phi_{{\lambda}}(E)$ and ${\lambda}\mapsto\eta_{{\lambda}}(r)$ are increasing for all $E>V_{\tilde{\lambda}}^0$, $\tilde{\lambda}\geq 0$ and $r>1$.
	\end{rem}
	For $\lambda\geq 0$, the function $\eta_\lambda(r)$ serves to express the total energy $E>V^0_\lambda$ of a solution of $\qty(\hyperlink{ODE}{\textsc{ode}})_{\lambda}$ as a function of $r>1$. This enables us to give an explicit formula for $\Theta_\lambda(r)$.
	\begin{satz}[{$\Theta_\lambda(\cdot)$} as a function of $r$]
		For $\lambda\geq 0$, the formula\thlabel{Thetaasafunctionofr}
		\begin{equation}\label{semiperiod_reloded}
			\Theta_\lambda(r)=\int_1^r\qty[1-y^2+\frac{2\lambda(y-1)}{\eta_\lambda(r)}+\frac{2\log y}{\eta^2_\lambda(r)}]^{-\sfrac{1}{2}} \dd{y}\qcomma{r>1},
		\end{equation}
		holds.
	\end{satz}
	\begin{proof}
		At its critical points, $\kappa_{\lambda,E}$ assumes the extremal values $\kappa^\pm_\lambda(E)$ and the kinetic energy term in $\qty(\hyperlink{FI}{\textsc{fi}})_{\lambda,E}$ vanishes. Therefore
		\begin{equation}
			E=\frac{1}{2}\qty(\kappa^\pm_\lambda(E))^2-\lambda\kappa^\pm_\lambda(E) -\log\kappa^\pm_\lambda(E).\tag*{(\theequation)$^{\pm}$}\labeltarget{pm}\refstepcounter{equation}
		\end{equation}
		Subtracting (\hyperlink{pm}{\theequation})$^-$ from (\hyperlink{pm}{\theequation})$^+$ and substituting $\kappa^+_\lambda(E)=r\kappa^-_\lambda(E)$ gives
		\begin{equation*}
			\qty(\kappa_\lambda^-(E))^2\left(\frac{r^2}{2}-\frac{1}{2}\right)+\kappa_\lambda^-(E)\lambda\left(1-r\right)-\log r=0
		\end{equation*}
		which is equivalent to $p_{\lambda,r}[\kappa_\lambda^-]=0$. We infer that $\kappa_\lambda^-(E)=\eta_\lambda(r)$ which, in light of (\hyperlink{pm}{\theequation})$^{-}$ determines the relationship between $E$ and $r$. Finally, to arrive at \eqref{semiperiod_reloded}, we apply the substitution $\kappa= y\kappa_\lambda^-(E)$ to the integral expression \eqref{semiperiod}. Observing that
		\begin{align*}
			E-V_\lambda(\kappa)&=V_\lambda(\kappa^-_\lambda(E))-V_\lambda(y\kappa_\lambda^-(E))\\
			&=\frac{\qty(\kappa^-_\lambda(E))^2}{2}\qty(1-y^2+\frac{2\lambda(y-1)}{\eta_\lambda(r)}+\frac{2\log y}{\eta_\lambda(r)^2}),
		\end{align*}
		the formula stated in \eqref{semiperiod_reloded} follows immediately by canceling redundant terms.
	\end{proof}
	Using the properties established in \thref{prop_of_r_eta}, we obtain a direct proof of the following monotonicity result, providing an alternative to the approach in \cite[Theorem 3.8]{Cha17}.
	\begin{coro}[monotonicity of $\lambda\mapsto\Theta_\lambda(r)$]\thlabel{monotonicity_lambda}
		The map $\lambda\mapsto\Theta_\lambda(r)$ is increasing for any $r>1$.
	\end{coro}
	\begin{nota}[suppressed notation for $\kappa^\pm_\lambda(E)$ and $\kappa_\lambda^0$]
		From now on, whenever we are given $\lambda\geq 0$ and $E>V_\lambda^0$, we write $\kappa^\pm$ instead of $\kappa_\lambda^\pm(E)$. We shall also write $\kappa^0$ instead of $\kappa^0_\lambda$.
	\end{nota}
	We finish this section by presenting some easy technical observations that will prove useful later on.
	\begin{lemma}[various technical results on $\log r$]
		Let $\lambda\geq 0$ and $E>V_\lambda^0$. Then, we have the identity
		\begin{equation}
			2\log r=\qty(\kappa^+-\kappa^-)\qty(\kappa^++\kappa^--2\lambda)\label{logr_ident}
		\end{equation}
		as well as the inequalities
		\begin{equation}
			\log r >\frac{\kappa^+-\kappa^-}{\kappa^0}\label{logr_estimate1}
		\end{equation}
		and
		\begin{equation}
			\kappa^++\kappa^--2\kappa^0<2\qty(\frac{1}{\kappa^-}-\frac{1}{\kappa^0}).\label{logr_estimate2}
		\end{equation}
	\end{lemma} 
	\begin{proof}
		The identity \eqref{logr_ident} trivially follows after expanding and rearranging $V(\kappa^+)-V(\kappa^-)=0$.
		Next, combining \eqref{min_ident} with \thref{skewed} yields
		\begin{equation}
			\kappa^++\kappa^--2\lambda>\frac{2}{\kappa^0}
		\end{equation}
		and $\eqref{logr_estimate1}$ follows by applying \eqref{logr_ident} along with the fact that $\kappa^+-\kappa^->0$.
		Finally, employing \eqref{logr_ident} along with the concavity of the logarithm gives
		\begin{equation}
			\frac{1}{2}\qty(\kappa^+-\kappa^-)\qty(\kappa^++\kappa^--2\lambda)=\log r<r-1=\frac{\kappa^+-\kappa^-}{\kappa^-}
		\end{equation}
		and \eqref{logr_estimate2} then follows by dividing by $\kappa^+-\kappa^->0$ and employing \eqref{min_ident}.
	\end{proof}
	\newpage
	\section{Main results}
	Suppose $\lambda\geq 0$. In order to find \eqref{APCSF}-shrinkers, we are first looking for $\lambda$-curves. Denote
	\begin{equation}
		W_\lambda\coloneqq\frac{1}{\pi}\range\Theta_\lambda(\cdot),\qquad\lambda\geq 0.
	\end{equation}
	Then, owing to \thref{closeuplambda} \ref{closeuplambda1}, each rational value $q\in W_\lambda$ corresponds to at least one family of congruent closed $\lambda$-curves whose geometry is determined by $q$. We will therefore examine $W_\lambda$ in more detail.
	\begin{rem}[the Abresch--Langer case, again]	\thlabel{ALcase}
		In their classification statement, \thref{abresch-langer}, \cite{AbrLan86} showed that $\Theta_0(\cdot)$ maps $(1,\infty)$ bijectively to $W_0=\qty(\sfrac{1}{2},\sfrac{\sqrt{2}}{2})$. The boundary of the latter interval determines the inequality \eqref{ineq2} and each rational number in $W_0$ uniquely corresponds to a \eqref{CSF}-shrinker up to similarity.
	\end{rem}
	As a conclusion to this part, we will use the fact that $\Theta_\lambda(\cdot)$ is bijective (in fact, monotonically decreasing) together with the specific structure of $W_\lambda$ to show that any rational value in $(\sfrac{1}{2},1)$ admits precisely one \eqref{APCSF}-shrinker. This then completes the proof of \thref{main-thm}.
	\subsection{Monotonicity of $\Theta_\lambda(\cdot)$}\label{sec_mon}
	The following considerations are largely inspired by the procedure of \cite[§5]{And03} and establish a generalization thereof. First, we eliminate the dependence of the upper integration bound in \eqref{semiperiod_reloded} on $r>1$. To that end, put
	\begin{equation}
		v\coloneqq \frac{r^\beta -1}{2}z+\frac{r^\beta +1}{2}\qcomma{\text{with $\beta>1$ and $z\in[-1,1]$}}.\label{defv}
	\end{equation}
	It is easy to see that $z\mapsto v$ maps $[-1,1]$ bijectively onto $[1,r^\beta]$. 
	\begin{nota}[$\kappa^\pm$ as functions of $r$]
		In later considerations, it will be helpful to express this substitution in terms of the original integration variable $\kappa$ of \eqref{semiperiod}. Suppose that $\lambda\geq 0$ and $E>V_\lambda^0$. The proof of  \thref{Thetaasafunctionofr} shows that $\kappa^-_\lambda(E)=\eta_\lambda(r)$ and $\kappa^+_\lambda(E)=r\eta_\lambda(r)$ whenever $r=\Phi_\lambda(E)$. In particular, $V_\lambda(\eta_\lambda(r))=V_\lambda(r\eta_\lambda(r))$ holds. Since $\Phi_\lambda$ is a diffeomorphism, we may regard the extremal values as functions of $r$. By abuse of notation, we write $\eta=\eta_\lambda(r)$ and put
		\begin{equation}
			\kappa^-(r)\coloneqq\kappa^-_\lambda\qty(\Phi_\lambda^{-1}(r))=\eta,\quad
			\kappa^+(r)\coloneqq\kappa^+_\lambda\qty(\Phi_\lambda^{-1}(r))=r\eta,\quad
			\kappa(z)\coloneqq\eta\,v^{\sfrac{1}{\beta}}.
		\end{equation}
		For fixed $r>1$, the map $z\mapsto\kappa(z)$ is the composition of $y=v^{\sfrac{1}{\beta}}$ with the transformation $\kappa=y\eta$. Hence it is a strictly increasing bijection from $[-1,1]$ onto $[\kappa^-(r),\kappa^+(r)]=[\eta,r\eta]$. Finally, $\kappa(z)^\beta=\eta^\beta v$ together with \eqref{defv} yields the interpolation
		\begin{equation}
			\kappa(z)^\beta=\frac{1-z}{2}\,(\kappa^-)^\beta+\frac{1+z}{2}\,(\kappa^+)^\beta .\label{interpol}
		\end{equation}
	\end{nota}
	
	\begin{satz}[period-map with fixed integration]
		For $\lambda\geq 0$ and $r>1$ the formula
		\begin{equation}
			\Theta_\lambda(r)=\int_{-1}^{1}\frac{\dd{z}}{\sqrt{J_\lambda^\beta(r,z)}}\label{trueif}
		\end{equation}
		holds with
		\begin{equation}
			J_\lambda^\beta(r,z)=\frac{4\beta^2}{(r^\beta -1)^2}\qty[v^{\frac{2(\beta-1)}{\beta}}-v^2+\frac{2\lambda}{\eta}\qty(v^{\frac{2\beta-1}{\beta}}-v^{\frac{2(\beta-1)}{\beta}})+\frac{2}{\beta\eta^2}v^{\frac{2(\beta-1)}{\beta}}\log v]
		\end{equation}
		with $z\in[-1,1]$.
	\end{satz}
	\begin{proof}
		Put $y\coloneqq v^{\sfrac{1}{\beta}}$. Then, comparing with \eqref{semiperiod_reloded} gives
		\begin{align}
			1-y^2+\frac{2\lambda(y-1)}{\eta}+\frac{2\log y}{\eta^2}=1-v^{\sfrac{2}{\beta}}+\frac{2\lambda}{\eta}\qty(v^{\sfrac{1}{\beta}}-1)+\frac{2}{\eta^2}\log v^{\sfrac{1}{\beta}}.
		\end{align}
		Moreover,
		\begin{equation*}
			\dv{y}{z}= \frac{r^\beta-1}{2\beta}v^{\frac{1-\beta}{\beta}}=\qty[\frac{4\beta^2}{(r^\beta-1)^2}v^{\frac{2(\beta-1)}{\beta}}]^{-\sfrac{1}{2}}
		\end{equation*}
		and \eqref{trueif} follows.
	\end{proof}
	Note that \eqref{trueif} in particular implies that
	\begin{equation}
		\dv{\Theta_\lambda(r)}{r}=-\frac{1}{2}\int_{-1}^1 \pdv{J_\lambda^\beta(r,z)}{r} \frac{\dd{z}}{\qty(J_\lambda^\beta(r,z))^{\sfrac{3}{2}}}
	\end{equation}
	Thus, in order to prove that $\Theta_\lambda(\cdot)$ is monotone, it suffices to show that $\partial_r J_\lambda^\beta(r,z)$ does not change its sign in $z\in(-1,1)$. Let us consider its marginal behavior first.
	\begin{lemma}[$J_\lambda^\beta(r,z)$ vanishes at $z=\pm1$]
		For $\lambda\geq 0$, $r>1$ and $\beta>1$ we have $\partial_rJ_\lambda^\beta(r,\pm1)=0$.
	\end{lemma}
	\begin{proof}
		We show that $J_\lambda^\beta(\cdot,\pm1)\equiv 0$. Suppose that $z=-1$. Then $v=1$ and the conclusion is trivial. If $z=1$, then $v=r^\beta$ and it follows that
		\begin{align}
			J^\beta_\lambda(r,1)&=\frac{4\beta^2}{(1-r^\beta)^2}\frac{r^{2(\beta-1)}\qty(1-r^2)}{\eta^2}\qty[\eta^2-\frac{2\lambda}{r+1}\eta-\frac{2\log r}{r^2-1}]\\
			&=\frac{4\beta^2}{(1-r^\beta)^2}\frac{r^{2(\beta-1)}\qty(1-r^2)}{\eta^2}p_{\lambda,r}\qty[\eta]=0
		\end{align}
		referring back to \thref{defining_property}.
	\end{proof}
	The preceding observation gives us a sufficient condition to show that $\partial_r J^\beta_\lambda\geq 0$. Given $r>1$ and $\lambda\geq 0$, assume that $\partial_r J^\beta_\lambda(r,\cdot)$ is strictly concave on $(-1,1)$ for some choice of $\beta>1$. Then the previous lemma implies $\partial_r J^\beta_\lambda(r,z)>0$. Thus, it suffices to show that
	\begin{equation}
		\pdv[2]{z}\pdv{J^\beta_\lambda(r,z)}{r}<0,\quad{\text{i.e. that}}\quad {\pdv[2]{J^\beta_\lambda(\cdot,z)}{z}}\quad\text{is decreasing.}\label{explicit}
	\end{equation}
	We thus consider $\eqref{explicit}$ explicitly.
	\begin{satz}[2$^\text{nd}$ $z$-derivative of $J_\lambda^\beta$]\thlabel{2nd-z-der}
		If $\lambda\geq0$, $r>1$, $\abs{z}<1$ and $\beta>1$ with $\beta\neq 2$, we have
		\begin{equation}
			\frac{1}{\beta^2}\pdv[2]{J_\lambda^\beta}{z}=-2-\frac{2(\beta-1)(2-\beta)}{\beta^2}Q\label{z-deriv}
		\end{equation}
		with
		\begin{equation}
			Q\coloneqq\frac{(4-3\beta)+2(\beta-1)(2-\beta)\qty(\log\kappa +E)}{(\beta-1)(2-\beta)\kappa^2}-\frac{\lambda(\beta-1)(2\beta-1)}{(\beta-1)(2-\beta)\kappa}.\label{Q}
		\end{equation}
	\end{satz}
	\begin{nota}[constants depending on $\beta$ and $\lambda$]\thlabel{constantss}
		In light of the preceding proposition, we introduce the constants
		\begin{equation}
			a\coloneqq4-3\beta,\quad b\coloneqq\lambda(\beta-1)(2\beta-1)\quad\text{and}\quad \mu\coloneqq2(\beta-1)(2-\beta)\label{constants}
		\end{equation}
		so that \eqref{z-deriv} and \eqref{Q} simplify to
		\begin{equation}
			\frac{1}{\beta^2}\,\partial_z^2J^\beta_\lambda=-2-\frac{\mu}{\beta^2}\,Q\quad\text{and}\quad Q=\frac{2}{\mu}\qty[\frac{a+\mu(\log\kappa+E)}{\kappa^2}-\frac{b}{\kappa}].
		\end{equation}
		Additionally, we put $c\coloneqq \qty(\kappa^0)^2$.
	\end{nota}
	\begin{proof}[Proof of \thref{2nd-z-der}]
		Observe that $\partial_zv=\frac{r^\beta-1}{2}$. Applying standard differentiation rules to $J^\beta_\lambda$ thus gives
		\begin{equation}
			\pdv[2]{J^\beta_\lambda}{z}=-2\beta^2+\frac{2b}{\eta}v^{-\frac1\beta}-\qty(\mu-\frac{2\lambda\mu}{\eta}+\frac{2\mu}{\beta\eta^2}\log v+\frac{2a}{\eta^2})v^{-\frac2\beta}.
		\end{equation}
		Applying the substitution $\kappa(z)= \eta v(z)^{\sfrac{1}{\beta}}$ to the equation above finally gives
		\begin{equation}
			\pdv[2]{J^\beta_\lambda}{z}=-2\beta^2+\frac{2b}{\kappa}-\frac{2}{\kappa^2}\qty[a+\mu\qty(\log\kappa+\frac{\eta^2}{2}-\lambda\eta-\log\eta)]=-2\beta^2-\mu Q,
		\end{equation}
		where we used that $\frac{\eta^2}{2}-\lambda\eta-\log\eta=V(\eta)=E$. 
	\end{proof}
	All previous results are true for general choices of $\beta>1$. We now fix a specific choice, which turns out to depend on the parameter $\lambda\geq 0$. Generalizing the procedure presented in \cite{And03}, it can be shown that $J_\lambda^\beta$ expands to
	\begin{equation}
		J^\beta_\lambda(r,z)=\frac{c+1}{c}\qty(1-z^2)\qty[1+\frac{z}{2}(\beta-\beta_\lambda)(r-1)]+\mathcal{O}\qty((r-1)^2)\quad\text{as $r\searrow 1$}
	\end{equation}
	with
	\begin{equation}
		\beta_\lambda=1+\frac{2}{3(c+1)}=\frac{4}{3}-\frac{\lambda}{3\sqrt{\lambda^2+4}}\qcomma{\text{for all $\lambda\geq 0$}}.\label{asdef}
	\end{equation}
	This justifies the following convention.
	\begin{nota}[choice of $\beta$]
		From hereon out, whenever we consider $\lambda\geq 0$, we assume $\beta\coloneqq \beta_\lambda$ as defined in \eqref{asdef}.
	\end{nota}
	With our fixed choice for the parameter $\beta$ we easily verify the following technical identity on the constants defined in \thref{constantss}.
	\begin{lemma}[properties of constants]\thlabel{constantsss}
		Let $\lambda\geq 0$. Then 
		\begin{equation}
			2a=\kappa^0(b+\mu\lambda).\label{constantssss}
		\end{equation}
		Moreover $a,b\geq 0$ as well as $c,\lambda>0$.
	\end{lemma}
	In fact, the equality stated in the equation above is false for any choice of $\beta$ different from \eqref{asdef}. The following decomposition chiefly relies on this identity, which itself is instrumental for the remaining argument.
	\begin{dis}
		The following few results, necessary to prove that $\partial_r Q\geq 0$ namely \thref{decom}, \thref{diff_ineq} and \thref{margin} were in large part obtained with the help of A.I.
	\end{dis}
	\begin{theo}[decomposition of $\partial_r Q$]\thlabel{decom}
		Suppose that $\lambda\geq 0$ and $r>1$. Then the identity
		\begin{equation}
			\frac{\mu\kappa^{\beta+2}}{2}\pdv{Q}{r}=\frac{1-z}{2}\,(\kappa^-)^{\beta-1}\dv{\kappa^-}{r}\,R^-(\kappa)+\frac{1+z}{2}(\kappa^+)^{\beta-1}\dv{\kappa^+}{r}R^+(\kappa)
		\end{equation}
		holds with
		\begin{equation}
			R^\pm(\kappa)=(b+\lambda\mu)\,(\kappa-\kappa^0)-\lambda\mu\,(\kappa-\kappa^\pm)-2\mu\log\frac{\kappa}{\kappa^\pm}.
		\end{equation}
		for all $\kappa\in[\kappa^-,\kappa^+]$.
	\end{theo}
	\begin{proof}
		Straightforward differentiation of $Q$ along with a multiplication by $\kappa^{\beta+2}$ leads to
		\begin{equation}
			\frac{\mu\kappa^{\beta+2}}{2}\pdv{Q}{r}=\qty(b\kappa+\mu-2a-2\mu\log\kappa-2\mu E)\,\kappa^{\beta-1}\pdv{\kappa}{r}+\mu\kappa^{\beta}\dv{E}{r}.\label{sub3}
		\end{equation}
		To determine the remaining derivatives, we first recall \eqref{interpol} and find
		\begin{equation}
			\kappa^{\beta-1}\pdv{\kappa}{r}=\frac{1-z}{2}(\kappa^-)^{\beta-1}\dv{\kappa^-}{r}+\frac{1+z}{2}(\kappa^+)^{\beta-1}\dv{\kappa^+}{r}.\label{sub1}
		\end{equation}
		Secondly, we use $E=V_\lambda(\kappa^\pm)$ to obtain
		\begin{equation}
			\dv{E}{r}= \eval{\dv{V_\lambda}{\kappa}}_{\kappa=\kappa^\pm}\dv{\kappa^\pm}{r}.\label{sub2}
		\end{equation}
		We substitute \eqref{sub1} and \eqref{sub2} into \eqref{sub3} while expanding $\kappa^\beta$ by \eqref{interpol}. Summing the right-hand side over both signs gives
		\begin{equation}
			\sum_{\pm}\frac{1\pm z}{2}\,(\kappa^\pm)^{\beta-1}\dv{\kappa^\pm}{r}\qty[b\kappa+\mu-2a-2\mu\log\kappa-2\mu E+\mu\,\kappa^\pm\eval{\dv{V_\lambda}{\kappa}}_{\kappa=\kappa^\pm}].
		\end{equation}
		The proof concludes after noticing that $V_\lambda(\kappa^\pm)=E$ implies
		\begin{equation}
			-2\mu E+\mu\kappa^\pm\eval{\dv{V_\lambda}{\kappa}}_{\kappa=\kappa^\pm}=\mu\qty(\lambda\kappa^\pm-1+2\log\kappa^\pm)
		\end{equation}
		and employing \eqref{constantssss}.
	\end{proof}
	Crucially, \thref{decom} gives us a sufficient condition to show $\partial_r J^\beta_\lambda$ is strictly concave. Since $E\mapsto\Phi_\lambda(E)$ is clearly increasing, we easily verify
	\begin{equation}
		\dv{\kappa^-}{r}<0<\dv{\kappa^+}{r}.
	\end{equation}
	Hence it suffices to show that $R^-<0<R^+$ to prove $\partial_r Q\geq0$. By \eqref{z-deriv} and \thref{constantsss} concavity immediately follows. To show this, we first consider further derivatives of $R^\pm$. In order to follow this strategy, we prepare the following differential inequality which holds for both $R^\pm$ although we will only use it for $R^-$.
	\begin{lemma}[absence of interior maxima for $R^\pm$]\thlabel{diff_ineq}
		Given $\abs{z}<1$ and $r>1$, we have
		\begin{equation}
			\pdv{z}\qty(v^{\frac{\beta-1}{\beta}}\pdv{R^\pm}{z})=\frac{\mu(r^\beta-1)^2}{2\beta^2}v^{-\frac{\beta+1}{\beta}}>0
		\end{equation}
		In particular, $z \mapsto R^\pm(\kappa(z))$ cannot have interior maxima.
	\end{lemma}
	\begin{proof}
		Elementary differentiation along with $\eta/\kappa= v^{-\sfrac{1}{\beta}}$ gives
		\begin{equation}
			\pdv{R^\pm}{z}=\qty(b-\frac{2\mu}{\kappa})\frac{r^\beta-1}{2\beta}\,\eta\,v^{\frac{1-\beta}{\beta}}=\frac{r^\beta-1}{2\beta}\qty(b\eta-2\mu\,v^{-\frac1\beta})v^{\frac{1-\beta}{\beta}}.
		\end{equation} 
		Therefore
		\begin{equation}
			\pdv{z}\qty(v^{\frac{\beta-1}{\beta}}\pdv{R^\pm}{z})=\pdv{z}\qty(\frac{r^\beta-1}{2\beta}\qty(b\eta-2\mu\,v^{-\frac1\beta}))=\frac{\qty(r^\beta-1)^2}{4\beta}\cdot\frac{2\mu}{\beta}v^{-\frac1\beta-1},
		\end{equation}
		which is clearly positive. If $R^\pm$ were to have a maximum at some $z^\ast\in(-1,1)$, we would find
		\begin{equation}
			\eval{\pdv{R^\pm}{z}}_{z=z^\ast}=0\quad\text{and}\quad\eval{\pdv[2]{R^\pm}{z}}_{z=z^\ast}\leq 0.
		\end{equation} 
		But then
		\begin{equation}
			\eval{\pdv{z}\qty(v^{\frac{\beta-1}{\beta}}\pdv{R^\pm}{z})}_{z=z^\ast}=v^{\frac{\beta-1}{\beta}}\eval{\pdv[2]{R^\pm}{z}}_{z=z^\ast}\leq 0,
		\end{equation}
		which is impossible.
	\end{proof}
	\begin{satz}[positivity of $\pm R^{\pm}$]\thlabel{margin}
		Given any $\lambda\geq 0$ and $\kappa\in(\kappa^-,\kappa^+)$ we have $\pm R^\pm(\kappa) >0$.
	\end{satz}
	\begin{proof}
		\begin{description}[labelindent=0cm, leftmargin=0cm]
			\item[positivity of $R^+$] If $\kappa\in\left[\kappa^0,\kappa^+\right)$, the claim is trivial as each term contributes non-negatively to $R^+$. Assume therefore $\kappa\in\qty(\kappa^-,\kappa_\lambda^0)$. Then, remembering \thref{skewed}, we find
			\begin{equation}
				\kappa^+-\kappa>2\kappa^0-\kappa^--\kappa>2(\kappa^0-\kappa).\label{onehand}
			\end{equation}
			On the other hand,
			\begin{equation}
				\log\qty(\frac{\kappa^+}{\kappa})>\log\qty(\frac{\kappa^0}{\kappa})\geq 1-\frac{\kappa}{\kappa^0}=\frac{\kappa^0-\kappa}{\kappa^0}\label{onehand2}
			\end{equation}
			by basic logarithmic inequalities. Putting \eqref{onehand} and \eqref{onehand2} to use, we estimate
			\begin{equation}
				R^+>(b+\lambda\mu)\qty(\kappa-\kappa^0)+2\lambda\mu\qty(\kappa^0-\kappa)+\frac{2\mu(\kappa^0-\kappa)}{\kappa^0}=\qty(\kappa^0-\kappa)\qty[\lambda\mu+\frac{2\mu}{\kappa^0}-b].
			\end{equation}
			It therefore suffices to show that the bracketed factor is non-negative. Recalling \eqref{constants} and \eqref{asdef} we continue with
			\begin{align}
				\qty[\lambda\mu+\frac{2\mu}{\kappa^0}-b]\frac{\kappa^0}{\beta-1}&=2\lambda\kappa^0(2-\beta)+4(2-\beta)-\lambda\kappa^0(2\beta-1)\\
				&=4(2-\beta)+(c-1)(5-4\beta)\\
				&=\frac{(c-1)(3c-5)+4(3c+1)}{3(c+1)}
			\end{align}
			Noticing that the numerator $(c-1)(3c-5)+4(3c+1)=3c^2+4c+9$ is strictly positive thus concludes the argument.
			\item[negativity of $R^-$] We give a maximum principle-type argument. By \thref{diff_ineq} $R^-$ cannot have an interior maximum. To show negativity in $(\kappa^-,\kappa^+)$, it therefore suffices to show non-positivity for the marginal values $R^-(\kappa^\pm)$. For the left boundary value, we clearly have
			\begin{equation}
				R^-(\kappa^-)=(b+\lambda\mu)(\kappa^-_\lambda-\kappa_\lambda^0)\leq0
			\end{equation}
			For the right boundary, we find
			\begin{align}
				R^-(\kappa^+)=\lambda\mu(\kappa^--\kappa^0)+b(\kappa^+-\kappa^0)-2\mu\log r.\label{right_boundry}
			\end{align}
			after expanding the factor $b+\lambda\mu$ and canceling redundant terms. On one hand, we apply \eqref{logr_estimate1} along with $\kappa^-<\kappa^0$ to arrive at 
			\begin{equation}
				R^-\qty(\kappa^+)\leq b\qty(\kappa^+-\kappa^0)-2\mu\qty(\frac{\kappa^+-\kappa^0}{\kappa^0})=\qty(\kappa^+-\kappa^0)\qty[b-\frac{2\mu}{\kappa^0}].\label{alt1}
			\end{equation}
			On the other hand, rearranging \eqref{right_boundry} gives
			\begin{align}
				R^-\qty(\kappa^+)&=\mu\qty[\lambda\qty(\kappa^++\kappa^--2\kappa^0)-2\log r]+\qty(\kappa^+-\kappa^0)\qty[b-\lambda\mu]\\
				&\eqqcolon \mu\tilde{R}+\qty(\kappa^+-\kappa^0)\qty[b-\lambda\mu].
			\end{align}
			We claim that $\tilde{R}$ is negative. Indeed, after combining \eqref{logr_ident} with \eqref{min_ident}, we first find 
			\begin{equation}
				2\log r=(\kappa^+-\kappa^-)(\kappa^++\kappa^--2\kappa^0)+\frac{2(\kappa^+-\kappa^-)}{\kappa^0}
			\end{equation}
			which leads to
			\begin{equation}
				\tilde{R}=\qty(\kappa^++\kappa^--2\kappa^0)\qty(\lambda-(\kappa^+-\kappa^-))-\frac{2(\kappa^+-\kappa^-)}{\kappa^0}.
			\end{equation}
			If $\kappa^+-\kappa^-\geq\lambda$, then $\tilde{R}$ is trivially negative. Assume therefore $\kappa^+-\kappa^-<\lambda$. Then, we apply \eqref{logr_estimate2} and find
			\begin{align}
				\tilde{R}&<\frac{2(\kappa^0-\kappa^-)}{\kappa^0\kappa^-}\qty(\lambda-\qty(\kappa^+-\kappa^-))-\frac{2(\kappa^+-\kappa^-)}{\kappa^0}\\
				&=\frac{2}{\kappa^0\kappa^-}\qty[\qty(\kappa^0-\kappa^-)\lambda-\qty(\kappa^+-\kappa^-)\kappa^0]<0
			\end{align}
			because $\kappa^+-\kappa^->\kappa^0-\kappa^->0$ and clearly $\kappa^0>\lambda$ (see e.g.~\thref{def_V})\ref{def_k_0}. Now, we use the negativity of $\tilde{R}$ to find the estimate
			\begin{equation}
				R^-\qty(\kappa^+)<\qty(\kappa^+-\kappa^0)\qty[b-\lambda\mu].\label{alt2}
			\end{equation}
			To finish the proof, we need to show that either of the bracketed terms in \eqref{alt1} or \eqref{alt2} is negative. A straightforward computation along with the observation that $\lambda\kappa^0(\beta-1)=\frac{2}{3}(4-3\beta)$ (recall \thref{constantsss}) shows that
			\begin{equation}
				\begin{cases}
					b\kappa^0-2\mu=\frac{2}{3}(4-3\beta)(2\beta-1)-4(\beta-1)(2-\beta)=\frac{2}{3}(8-7\beta),\\
					b-\lambda\mu=\lambda(\beta-1)\qty((2\beta-1)-2(2-\beta))=\lambda(\beta-1)(4\beta-5).
				\end{cases}
			\end{equation}
			If $\beta\geq \frac{8}{7}$, then $b\kappa^0-2\mu\leq 0$ and the negativity of $R^-(\kappa^+)$ follows by \eqref{alt1}. If $\beta< \frac{8}{7}$, then in particular $\beta< \frac{5}{4}$ and $\lambda>0$. Recalling that $\beta>1$ it follows that $b-\lambda\mu<0$ and $R^-(\kappa^+)$ is negative by $\eqref{alt2}$.
		\end{description}
	\end{proof}
	By putting our previous thoughts together, the monotonicity of the period map becomes immediate. This also confirms a conjecture by J.-E. Chang, see \cite[§4]{Cha17}.
	\begin{coro}[monotonicity of the period map]\thlabel{montone}
		The map $E\mapsto\Theta_\lambda(E)$ decreases on $(V^0_\lambda,\infty)$ for all $\lambda\geq 0$.
	\end{coro}
	\subsection{Asymptotics and range of $\Theta_\lambda(\cdot)$}\label{sec2}
	We want to show that given $\lambda\geq 0$, $W_\lambda$ is ``sufficiently large'' to guarantee the existence of \eqref{APCSF}-shrinkers. To that end, we will compute the limiting values of $E\mapsto\Theta_\lambda(E)$ for high and low energies. If we define
	\begin{equation}
		\omega^-_\lambda\coloneqq \lim\limits_{E\to\infty}\frac{\Theta_\lambda(E)}{\pi}\qand\omega^+_\lambda\coloneqq\lim\limits_{E\searrow V_\lambda^0}\frac{\Theta_\lambda(E)}{\pi},\qquad\lambda\geq 0,
	\end{equation}
	then, by \thref{montone}, $\qty(\omega^-_\lambda,\omega^+_\lambda)= W_\lambda$. Both of these limits have been studied independently by \cite{Cha17} and \cite{TsaWan18}. Surprisingly, the limit for $E\to\infty$ is independent of the parameter $\lambda\geq 0$ and therefore coincides with \thref{ALcase}.
	\begin{satz}[high-energy limit]
		Let $\lambda\geq 0$. Then $\omega^-_\lambda=\sfrac{1}{2}$.
	\end{satz}
	Inspired by the procedure of \cite[Lemma 5.8]{Urb99}, we present an argument that leverages the representation of $\Theta_\lambda(\cdot)$ by means of the amplitude ratio. For two alternative proofs, also see \cite[Proposition 3.6]{Cha17} or \cite[Lemma 35]{TsaWan18}.
	\begin{proof}
		We will show that the  assertion holds for the limit $r\to\infty$ instead of $E\to\infty$. First, note that
		\begin{equation}
			y\longmapsto1+\frac{2\lambda(y-1)}{\eta_\lambda(r)}+\frac{2\log y}{\eta^2_\lambda(r)},\quad\text{with $\lambda\geq0$ and $r>1$}\label{first_note}
		\end{equation}
		is increasing on $(1,r)$. Taking the limit $y\nearrow r$ thus gives the supremal value
		\begin{equation}
			1+\frac{2\lambda(r-1)}{\eta_\lambda(r)}+\frac{2\log r}{\eta^2_\lambda(r)}=\frac{1-r^2}{\eta^2_\lambda(r)}p_{\lambda,r}[\eta_\lambda(r)]+r^2=r^2.\label{supval}
		\end{equation}
		It then immediately follows that
		\begin{equation}
			\lim\limits_{r\to\infty}\Theta_\lambda(r)\geq\lim\limits_{r\to\infty}\int_{1}^{r}\frac{\dd{y}}{\sqrt{r^2-y^2}}=\lim\limits_{r\to\infty}\arcsec(r)=\frac{\pi}{2}.
		\end{equation}
		For the converse inequality we observe that the monotonicity of \eqref{first_note} along with its supremal value \eqref{supval} implies that given any $\varepsilon\in(0,1)$, there is a $\tilde{y}(\varepsilon)\in(1,r)$ such that 
		\begin{align}
			\label{second_note}
			&1+\frac{2\lambda(y-1)}{\eta_\lambda(r)}+\frac{2\log y}{\eta^2_\lambda(r)}\geq(1-\varepsilon)r^2\quad\\
			\notag 	&\text{for all $y\in \qty(\tilde{y}(\varepsilon),r)$,}\quad\text{and moreover}\quad \tilde{y}(\varepsilon)\xrightarrow{\varepsilon\searrow 0}1.
		\end{align}
		The consequent estimate
		\begin{align}
			&\int_{\tilde{y}(\varepsilon)}^{r}\qty[1-y^2+\frac{2\lambda(y-1)}{\eta_\lambda(r)}+\frac{2\log y}{\eta^2_\lambda(r)}]^{-\sfrac{1}{2}}\dd{y}\\
			&\qquad \leq\frac{1}{\sqrt{1-\varepsilon}}\int_{\tilde{y}(\varepsilon)}^{r}\frac{\dd{y}}{\sqrt{r^2-y^2}}\leq\frac{\pi}{2\sqrt{1-\varepsilon}},
		\end{align}
		then yields
		\begin{align}
			\lim_{r\to\infty}\Theta_\lambda(r)&\leq\frac{\pi}{2\sqrt{1-\varepsilon}}\\
			&\quad +\lim_{r\to\infty}\int^{\tilde{y}(\varepsilon)}_1\qty[1-y^2+\frac{2\lambda(y-1)}{\eta_\lambda(r)}+\frac{2\log y}{\eta_\lambda^2(r)}]^{-\sfrac{1}{2}}\dd{y}\xrightarrow{\varepsilon\searrow 0}\frac{\pi}{2},
		\end{align}
		where the remaining integral term vanishes due to dominated convergence.
	\end{proof}
	For the low-energy limit we refer back to the interpretation of the energy-period map as a physical travel time, \thref{phys_travel_time} \ref{phys_travel_time_2}, and employ the \textit{small oscillation formula} known from theoretical mechanics. Consider a periodic motion $\xi\in\rmC[0,\infty)$ of a particle in a convex potential $U\in\rmC^{2}(0,\infty)$ around a minimum $U^{0}$ located at $\xi^{0}>0$. If we let the initial energy $E=U(\xi(0))$ approach $U^{0}$, the motion becomes indistinguishable from a harmonic oscillation. This allows us to infer a general formula for the period map in the limit $E\searrow U^{0}$.
	\begin{lemma}[small oscillations formula; e.g.~{\cite[p. 20]{Arn89}}]
		Suppose $U\in\rmC^2(0,\infty)$ is strictly convex and has a minimum $U^{0}$ at $\xi^{0}>0$ such that $U^{\prime\prime}(\xi^0)>0$. For $E>U^{0}$, let $\xi^-(E)<\xi^+(E)$ be determined by $U(\xi^\pm(E))=E$. Then
		\begin{equation}
			\lim\limits_{E\searrow U^{0}} \int_{\xi^-(E)}^{\xi^+(E)}\frac{\dd{\xi}}{\sqrt{2(E-U(\xi))}}=\frac{\pi}{\sqrt{U^{\prime\prime}\qty(\xi^{0})}}.\label{smallosc}
		\end{equation}
	\end{lemma}
	Proving this result essentially amounts to expanding $U$ up to second order and arguing that higher orders vanish in the limit. 
	We refer to e.g.~\cite[Lemma 3.1]{Cha17} for an elegant argument. 
	Evaluating \eqref{smallosc} for the $\lambda$-potential $V_\lambda$ yields
	\begin{equation}
		\omega^+_\lambda=\frac{\sqrt{2}}{2}\sqrt{\frac{\lambda}{\sqrt{\lambda^2+4}}+1},\quad\lambda\geq 0.
	\end{equation}
	This is, of course, consistent with \thref{ALcase}. For the remaining considerations, the explicit formula for $\omega^+_\lambda$ is not as important as its asymptotic properties and monotonicity.
	\begin{coro}[low-energy limit]\thlabel{owingto}
		The map $\lambda\mapsto\omega_\lambda^+$ is strictly increasing and satisfies \begin{equation}
			\lim_{\lambda\searrow 0}\omega^+_\lambda=\frac{\sqrt{2}}{2}\qand\lim_{\lambda\to\infty}\omega^+_\lambda=1.
		\end{equation}
	\end{coro}
	In light of this observation, we put $W_\infty\coloneqq \qty(\sfrac{1}{2},1)$. Now, if $q\in W_\infty\cap\bbQ$, we are interested in all $\lambda\geq 0$ such that there is a $\lambda$-curve corresponding to an initial energy $E>V_\lambda^0$ satisfying
	\begin{equation}
		\frac{\Theta_\lambda(E)}{\pi}=q\label{determinedby}.
	\end{equation}
	In other words, given $q\in W_\infty\cap\bbQ$, we look for all $\lambda\geq 0$ for which
	\eqref{determinedby} has a solution $E>V^0_\lambda$. Since $W_\lambda=(\sfrac{1}{2},\omega^+_\lambda)$, this is
	the case if and only if $q<\omega^+_\lambda$. Owing to \thref{owingto}, this condition is, for
	$\lambda>0$, equivalent to $\lambda>\lambda_q$ with
	\begin{equation}\label{simto}
		\lambda_q\coloneqq\begin{rcases}
			\begin{dcases}
				0&\text{if $q<\frac{\sqrt{2}}{2}$,}\\
				\qty(\omega^+_\cdot)^{-1}(q)&\text{if $q>\frac{\sqrt{2}}{2},$}
			\end{dcases}		
		\end{rcases}
		=\begin{dcases}
			0&\text{if $q<\frac{\sqrt{2}}{2}$,}\\
			\frac{2q^2-1}{q\sqrt{1-q^2}}&\text{if $q>\frac{\sqrt{2}}{2}.$}
		\end{dcases}
	\end{equation}
	To formalize the notion that for each $q\in W_\infty\cap \bbQ$ and $\lambda>\lambda_q$ there is a closed $\lambda$-curve whose geometry is determined by $q$, we introduce a slight abuse of notation.
	\begin{nota}[Amendment to \thref{nota1}]
		Let $q\in W_\infty\cap\bbQ$ and $\lambda>\lambda_q$.
		We will write $\qty(\hyperlink{IVP}{\textsc{ivp}})_{\lambda,q}$ for the  initial value problem $\qty(\hyperlink{IVP}{\textsc{ivp}})_{\lambda,E}$ where $E>V^{0}_\lambda$ is the (by \thref{montone}) unique initial energy satisfying \eqref{determinedby}.
		The unique solution of $\qty(\hyperlink{IVP}{\textsc{ivp}})_{\lambda,q}$ will be denoted as $\kappa_{\lambda,q}$.
		As before, any member of the family of parametrized curves in tangential polar coordinates whose curvature function is $\kappa_{\lambda,q}$ will be denoted as $x_{\lambda,q}$. Finally, we also write $\gamma_{\lambda,q}\coloneqq x_{\lambda,q}(\bbR)$.
	\end{nota}
	In closing this section and as preparation for the upcoming proof of existence, we formally state an easy observation about how the geometric properties of $\gamma_{\lambda,q}$ depend on $q\in W_\infty\cap \bbQ$. This is analogous to and immediate from \thref{closeuplambda}.
	\begin{coro}[geometric properties of $\gamma_{\lambda,q}$]\thlabel{geomprop}
		Suppose $m,n\in\bbN$ are coprime with $\sfrac{m}{n}\in W_\infty$ and $\lambda>\lambda_{\sfrac{m}{n}}$. Then
		$x_{\lambda,\sfrac{m}{n}}$ is closed, non-circular and has tangent turning index $m$. Moreover, $\gamma_{\lambda,\sfrac{m}{n}}$ exhibits $n$-fold rotational symmetry.	
	\end{coro}
	\subsection{$\lambda_\ast$-curves}
	We now prove \thref{main-thm}. We will show that for each $q\in W_\infty \cap \bbQ$, i.e., for each coprime pair $m,n\in\bbN$ with $\sfrac{1}{2}<\sfrac{m}{n}<1$, there is a unique $\lambda_\ast>\lambda_q$ such that $x_{\lambda_\ast,q}$ is an \eqref{APCSF}-shrinker and, owing to \thref{geomprop}, has the required geometric properties. We put particular emphasis on such special $\lambda$-curves with the following terminology.
	\begin{defi}[$\lambda_\ast$-curves]
		Let $x:\bbS^1\to\bbR^2$ be a $\lambda$-curve. We then say that $x$ (or $x(\bbS^1)$) is a $\lambda_\ast$-curve if $\vol[x]=0$.
	\end{defi}
	\begin{figure}[h]
		\centering
		\begin{tikzpicture}[achse/.style={black!40, line width=0.4pt}]
			\def\kAbst{1.7}       % Abstand benachbarter Kurvenmitten
			\def\kRad{0.5798}     % halbe Kantenlaenge einer Kurve (Mittellinie)
			\def\kStart{0.9712}   % x-Koordinate der ersten Kurvenmitte
			\def\kHoehe{0.6584}   % Hoehe der Kurvenmitten ueber der Achse
			\pgfmathsetmacro\aLaenge{2*\kStart+5*\kAbst}        % Achse von 0 bis zur Pfeilspitze
			\pgfmathsetlengthmacro\kBreite{2*\kRad cm+0.4pt}    % Bildbreite inkl. Strichstaerke
			
			\foreach \datei [count=\i from 0] in {L50,L40,L30,L25,L19p5,L10}
			\node[inner sep=0pt] at ({\kStart+\i*\kAbst},\kHoehe)
			{\includegraphics[width=\kBreite]{pics/spec-kurven/\datei}};
			
			% ---- lambda-Achse mit Pfeilspitze ---------------------------------------------
			\draw[achse] (0,0) -- ({\aLaenge-0.134},0);
			\fill[black!40, shift={(\aLaenge,0)}]
			(0,0) .. controls (-0.0831,0.0480) and (-0.1077,0.0836) .. (-0.1762,0.1408)
			.. controls (-0.1785,0.0949) and (-0.1612,0.0372) .. (-0.134,0)
			.. controls (-0.1612,-0.0372) and (-0.1785,-0.0949) .. (-0.1762,-0.1408)
			.. controls (-0.1077,-0.0836) and (-0.0831,-0.0480) .. (0,0) -- cycle;
			\node[anchor=base east, inner sep=0pt] at ({\aLaenge-0.0197},0.2473) {$\lambda$};
			
			% ---- Klammer und Beschriftungen -----------------------------------------------
			\draw[line width=0.4pt, decorate, decoration={brace, mirror, amplitude=0.2719cm}]
			({\kStart-\kRad},-0.0819) -- ({\kStart+3*\kAbst+\kRad},-0.0819);
			\node[anchor=north, inner sep=0pt] at ({\kStart+1.5*\kAbst},-0.4045) {$\vol>0$};
			\node[anchor=north, inner sep=0pt] at ({\kStart+4*\kAbst},-0.0819) {$\vol=0$};
			\node[anchor=north, inner sep=0pt] at ({\kStart+5*\kAbst},-0.0819) {$\vol<0$};
		\end{tikzpicture}
		\caption{Closed $\lambda$-curves corresponding to $q=\sfrac{3}{4}$ with various values of $\lambda>\lambda_{3/4}$ (or equivalently $L\in(0,L_{3/4})$) along with the sign of their algebraic area.}\label{spectrum}
	\end{figure}
	Let $q\in W_\infty\cap \bbQ$ and $\lambda>\lambda_q$. Then, by virtue of the discussion in Section \ref{part1}, $\vol[x_{\lambda,q}]=0$ clearly amounts to $\bar{\kappa}\qty[x_{\lambda,q}]=\lambda$ and thus $x_{\lambda,q}$ being an \eqref{APCSF}-shrinker. Put
	\begin{equation}
		\bar{\kappa}(\lambda)\coloneqq\bar{\kappa}\qty[x_{\lambda,q}],\quad \lambda\in(\lambda_q,\infty).\label{kappabarlambda}
	\end{equation}
	Then, solving for $\lambda_\ast>\lambda_q$ is equivalent to finding all fixed points of $\lambda\mapsto\bar{\kappa}(\lambda)$ in $\qty(\lambda_{q},\infty)$. We claim that there is exactly one such value. As we would like to employ a geometric argument to prove this, we first determine the length and area of $x_{\lambda,q}$ as a function of $\lambda>\lambda_q$. Similarly to \eqref{simto}, we define for coprime $m,n\in\bbN$ with $\sfrac{m}{n}\in W_\infty$ and $q\coloneqq \sfrac{m}{n}$
	\begin{equation}
		L_q\coloneqq\begin{rcases}\begin{dcases}
				\infty&\text{if }q<\frac{\sqrt{2}}{2}\\
				\frac{2\pi m}{\lambda_q}&\text{if }q>\frac{\sqrt{2}}{2}
			\end{dcases}
		\end{rcases}=\begin{dcases}
			\infty&\text{if }\frac{m}{n}<\frac{\sqrt{2}}{2},\\
			\frac{2\pi m^2\sqrt{n^2-m^2}}{2m^2-n^2}&\text{if }\frac{m}{n}>\frac{\sqrt{2}}{2}.
		\end{dcases}
	\end{equation}
	Then, by abuse of notation, write
	\begin{equation}
		x_{L,q}\coloneqq x_{\frac{2\pi m}{L},q}\text{ along with }\kappa_{L,q}\coloneqq \kappa_{\frac{2\pi m}{L},q}, \quad L\in \qty(0,L_q).
	\end{equation}
	Note that the parameter $L\in(0,L_q)$ is not necessarily the length of $x_{L,q}$ since generally
	\begin{equation}\label{len-q}
		\len\qty[x_{L,q}]=\int_0^{2\pi m}\frac{\dd{\theta}}{\kappa_{L,q}}=\int_{0}^{2\pi m}\kappa_{L,q}\dd{\theta}-\frac{4\pi^2 m^2}{L}
	\end{equation}
	by \eqref{lenform} and $\qty(\hyperlink{ODE}{\textsc{ode}})_{\lambda}$. However, the cases where $L\in (0,L_q)$ happens to coincide with $\len[x_{L,q}]$ characterize $x_{L,q}$ as being an \eqref{APCSF}-shrinker (see also Figure \ref{spectrum}).
	\begin{lemma}[characterizing $\lambda_\ast$-curves]
		Suppose $m,n\in\bbN$ are coprime with $\sfrac{m}{n}\in W_\infty$ and $q\coloneqq \sfrac{m}{n}$. Then, for $L\in\qty(0,L_q)$, the following assertions are equivalent
		\begin{enumerate}
			\item $x_{L,q}$ is an \eqref{APCSF}-shrinker\label{1},
			\item $x_{L,q}$ is a $\lambda_\ast$-curve,\label{5}
			\item $\bar{\kappa}\qty[x_{L,q}]=\frac{2\pi m}{L}$\label{2},
			\item $\len\qty[x_{L,q}]=L$ \label{3},
			\item $\vol\qty[x_{L,q}]=0$ \label{4}.
		\end{enumerate}
	\end{lemma}
	\begin{proof}
		As mentioned, the equivalence of \ref{1}, \ref{5} and \ref{2} is obvious from the discussion in Section \ref{sec1}. Furthermore, by \thref{geomprop}, $x_{L,q}$ has tangent turning index $m$, showing that \ref{2} and \ref{3} are equivalent.
		Finally, we integrate \eqref{lamdacurveequation} with $\lambda=(2\pi m)/L$ and obtain
		\begin{equation}
			\vol\qty[x_{L,q}]=m\pi-\frac{m\pi}{L}\len\qty[x_{L,q}],
		\end{equation}
		which confirms the equivalence of \ref{3} and \ref{4}.
	\end{proof}
	Due to the above observation, the number of fixed points of \eqref{kappabarlambda} can be determined by the fixed points of 
	\begin{equation}
		\len_q(L)\coloneqq \len[x_{L,q}],\quad L\in(0,L_q).
	\end{equation}
	The last theorem of this section and the accompanying corollary confirm that this map indeed has a unique fixed point. \thref{main-thm} then follows as an immediate consequence.
	\begin{theo}[properties of $L\mapsto\len_q(L)$]\thlabel{pro}
		Let $m,n\in\bbN$ be coprime with $\sfrac{m}{n}\in W_\infty$ and $q\coloneqq \sfrac{m}{n}$.
		\begin{enumerate}
			\item The limit $\len_q(L)\to\infty$ holds as $L\searrow 0$.\label{pro2}
			\item We have the limit\label{pro3}
			\begin{equation}
				\lim_{L\nearrow L_q}\len_q(L)=\begin{dcases}
					2\pi m&\text{if }q<\frac{\sqrt{2}}{2},\\
					\frac{2\pi m}{\kappa^0_{\sfrac{2\pi m}{L_q}}}&\text{if }q>\frac{\sqrt{2}}{2}.
				\end{dcases}
			\end{equation}
			\item The map $L\mapsto\len_q(L)$ is strictly monotonically decreasing.\label{pro1}
		\end{enumerate}
	\end{theo}
	\begin{proof}\leavevmode
		\begin{enumerate}
			\item
			Recall that $\gamma_{L,q}\coloneqq x_{L,q}(\bbR)$ exhibits $n$-fold rotational symmetry, i.e., that $\kappa_{L,q}$ is periodic with minimal period $\sfrac{2\pi m}{n}$ and is even around each of its interior extremal points. In particular
			\begin{equation*}
				\len_q(L)=2n\int_{0}^{\sfrac{\pi m}{n}}\frac{\dd{\theta}}{\kappa_{L,q}}\eqqcolon 2n\ell,\quad L\in\qty(0,L_q).
			\end{equation*}
			Geometrically, $\ell$ is the intrinsic distance of two consecutive vertices $S_1,S_2\in\gamma_{L,q}$. We denote $d\coloneqq\abs{S_1-S_2}$ as the corresponding extrinsic distance. Next, draw a perpendicular from $S_1$ to the radial line passing through $S_2$ and denote this distance as $\Delta$ (see Figure \ref{illustration}, left). Since clearly $d\geq \Delta$, it suffices to show that $\Delta\to\infty$ as $L\searrow 0$. To that end, denote the circumradius of $\gamma_{L,q}$ as $R$. Then, by basic trigonometry (see Figure \ref{illustration}, right)
			\begin{equation*}
				\Delta=\sin(\pi m/n)R.
			\end{equation*}
			\begin{figure}
				\centering
				\begin{tikzpicture}[
					every node/.style={inner sep=0pt},
					pfeil/.tip={Stealth[length=1.1pt, width=1.7pt, inset=0.2pt]},   % duenne Linien
					Pfeil/.tip={Stealth[length=2pt, width=3.2pt, inset=0.4pt]},     % punktierte Linien
					duenn/.style={line width=0.4pt, pfeil-pfeil},
					punkte/.style={line width=0.8pt, dash pattern=on 0.8pt off 0.8pt, Pfeil-Pfeil},
					]
					\def\Rad{3.85}                                     % Umkreisradius R in cm
					\pgfmathsetlengthmacro\strich{pi*\Rad cm/88}       % 88 Striche auf dem Kreis, 22 pro Viertel
					\pgfmathsetlengthmacro\phase{1.5*\strich}          % Luecke mittig auf den Beruehrpunkten
					\pgfmathsetlengthmacro\kurvenbreite{2*\Rad cm+0.8pt}
					\tikzset{kreis/.style={line width=0.8pt, dash pattern=on \strich off \strich}}
					
					\begin{scope}[x=\Rad cm, y=\Rad cm]
						% ---- linke Abbildung ------------------------------------------------------
						\coordinate (O)  at (0,0);
						\coordinate (S1) at (0,1);
						\coordinate (S2) at (0.2168,0.2168);
						\coordinate (F)  at (0.5,0.5);                    
						
						\node at (O) {\includegraphics[width=\kurvenbreite]{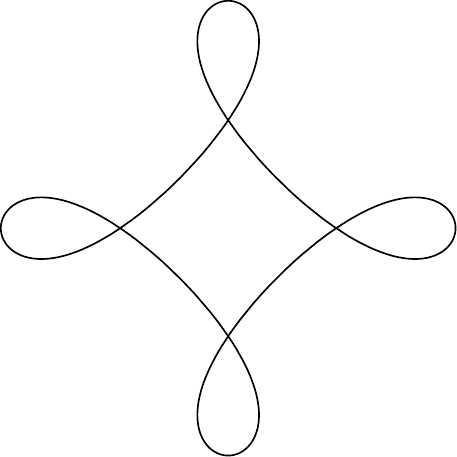}};
						\draw[kreis, dash phase=\phase] (O) circle[radius=1];
						\foreach \w in {0,90,180,270}
						\draw[line width=0.8pt] (\w-3.068:1) arc[start angle=\w-3.068, end angle=\w+3.068, radius=1];
						
						\draw[black!30, line width=0.4pt] (-0.011,-0.011) -- (0.011,0.011) (-0.011,0.011) -- (0.011,-0.011);
						\draw[punkte, shorten <=1.2pt, shorten >=1.6pt] (O) -- (45:1);
						\draw[duenn, shorten <=0.4pt, shorten >=0.4pt] (S1) -- (S2);
						\draw[duenn, shorten <=0.4pt, shorten >=0.4pt] (S1) -- (F);
						\draw[duenn] (0,1.035)
						.. controls (-0.0985,1.0350) and (-0.1587,0.9475) .. (-0.1687,0.8559)
						.. controls (-0.1785,0.7652) and (-0.1479,0.6749) .. (-0.1100,0.5941)
						.. controls (-0.0452,0.4561) and (0.0846,0.2973) .. (0.1923,0.1918);
						
						\node[anchor=base]  at (0,1.06)          {$S_1$};
						\node[anchor=west]  at (0.2618,0.2108)   {$S_2$};
						\node               at (-0.185,0.680)    {$\ell$};
						\node               at (0.178,0.543)     {$d$};
						\node               at (0.3090,0.7560)   {$\Delta$};
						\node               at (0.4306,0.3430)   {$R$};
						\node[anchor=south] at (0.758,0.148)     {$\gamma_{L,q}$};
						
						% ---- rechte Abbildung: Dreieck O S_1 S_2, um -45 Grad gedreht ----------------
						\begin{scope}[shift={(1.25,-0.35)}]
							\coordinate (O')  at (0,0);
							\coordinate (S1') at (45:1);
							\coordinate (S2') at (0.3066,0);
							\coordinate (F')  at (0.7071,0);
							
							\draw[kreis] (1,0) arc[start angle=0, end angle=45, radius=1];
							\draw[punkte, black!40] (O') -- (S1');
							\draw[punkte, shorten >=1.6pt] (O') -- (1,0);
							\draw[duenn, shorten >=0.4pt] (S1') -- (F');
							\draw[duenn, shorten >=0.4pt] (S1') -- (S2');
							
							\node[anchor=base]  at (0.7071,0.7376)  {$S_1$};
							\node[anchor=north] at (0.3106,-0.019)  {$S_2$};
							\node               at (0.7621,0.354)   {$\Delta$};
							\node               at (0.5176,0.2904)  {$d$};
							\node[anchor=north] at (0.489,-0.019)   {$R$};
						\end{scope}
					\end{scope}
				\end{tikzpicture}
				\caption{Illustration of the proof of \thref{pro} \ref{pro2}.}\label{illustration}
			\end{figure}
			On the other hand, we can also express $R$ analytically. Denote by $\sigma_{L,q}$ the support function of $x_{L,q}$. Referring back to \thref{refback} \ref{comon}, $\sigma_{L,q}$ attains the circumradius at each angle at which the curvature is maximal, (e.g.~at $\theta^+=m\pi/n$). By recalling \thref{extrem} we conclude
			\begin{equation*}
				R=\sigma_{L,q}\qty(\frac{\pi m}{n})=\kappa^+_{\sfrac{2\pi m}{L}}(E)-\frac{2\pi m}{L}\xrightarrow{L\searrow 0}\infty.
			\end{equation*}
			\item
			First, consider the case $q>\sfrac{\sqrt{2}}{2}$. We claim that $x_{L,q}$ becomes circular as $L\nearrow L_q$, i.e., that $\kappa_{L,q}$ converges uniformly to a constant function. Indeed, a straightforward calculation shows that
			\begin{equation}
				\kappa^0_{\sfrac{2\pi m}{L_q}}=\frac{\pi m}{L_q}+\sqrt{\frac{\pi^2 m^2}{L_q^2}+1}\quad\text{satisfies}\quad\frac{1}{\kappa^0_{\sfrac{2\pi m}{L_q}}}-\kappa^0_{\sfrac{2\pi m}{L_q}}+\frac{2\pi m}{L_q}=0.
			\end{equation}
			Thus, $\kappa^0_{\sfrac{2\pi m}{L_q}}$ is a constant solution of $\qty(\hyperlink{ODE}{\textsc{ode}})_{2\pi m/L_q}$. We then infer from \ODE-uniqueness and continuous dependence on parameters that $x_{L,q}$ converges to an $m$-fold cover of a circle of radius $1/\kappa^0_{\sfrac{2\pi m}{L_q}}$. The argument for the case $q<\sfrac{\sqrt{2}}{2}$ is completely analogous.
			\item Put $V_L\coloneqq V_{L}$ and $V^0_L\coloneqq V^0_{2\pi m/L}$ and define the maps
			\begin{equation}
				A_L(E)\coloneqq\int_{\kappa^-}^{\kappa^+}\frac{\sqrt{2\qty(E-V_L(\kappa))}}{\kappa}\dd{\kappa}\quad\text{and}\quad B_L(E)\coloneqq \int_{\kappa^-}^{\kappa^+}\sqrt{2\qty(E-V_L(\kappa))}\dd{\kappa}
			\end{equation}
			for $E>V_{L}^0$. Notice that, for both $A_L$ and $B_L$, the integrand vanishes at $\kappa=\kappa^\pm$. We thus find the derivatives
			\begin{equation}
				\pdv{A_L}{E}=\int_{\kappa^-}^{\kappa^+}\frac{\dd{\kappa}}{\kappa\sqrt{2(E-V_L)}}=\int_0^{m\pi/n}\frac{\dd{\theta}}{\kappa}=\frac{\len_q(L)}{2n}
			\end{equation}
			and
			\begin{equation}
				\pdv{A_L}{L}=-\frac{2\pi m}{L^2}\int_{\kappa^-}^{\kappa^+}\frac{\dd{\kappa}}{\sqrt{2(E-V_L)}}=-\frac{2\pi m\,\Theta_{2\pi m/L}(E)}{L^2}.
			\end{equation}
			In a similar fashion we obtain $\pdv*{B_L}{E}=\Theta_{2\pi m/L}(E)$ and
			\begin{equation}
				\pdv{B_L}{L}=-\frac{2\pi m}{L^2}\int_0^{\Theta_{2\pi m/L}(E)}\kappa\dd{\theta}=-\frac{2\pi m}{L^2}\qty(\frac{2\pi m}{L}\Theta_{2\pi m/L}(E)+\frac{\len_q(L)}{2n}),
			\end{equation}
			where the final equation uses $\qty(\hyperlink{ODE}{\textsc{ode}})_{2\pi m/L}$ along with $\kappa'(0)=\kappa'\qty(\Theta_{2\pi m/L}(E))=0$.
			Commuting derivatives and the implicit function theorem gives
			\begin{equation}
				\pdv{\len_q}{L}=\frac{4\pi mn}{L^2}\qty(\pdv{\Theta_{2\pi m/L}}{E})^{-1}C^-C^+,\label{sign}
			\end{equation}
			where
			\begin{equation}
				C^\pm\coloneqq \frac{1}{2n}\pdv{\len_q}{E}\pm\qty(\kappa^0)^{\pm1}\pdv{\Theta_{2\pi m/L}}{E}=\pdv{E}\qty(\frac{\len_q(L)}{2n}\pm\qty(\kappa^0)^{\pm1}\Theta_{2\pi m/L}).
			\end{equation}
			We consider $C^\pm$ separately. On one hand, integration by parts leads to
			\begin{align}
				\frac{\len_q(L)}{2n}-\frac{\Theta_{2\pi m/L}}{\kappa^0}&=\int_{\kappa^-}^{\kappa^+}\qty(\frac{1}{\kappa}-\frac{1}{\kappa^0})\frac{\dd{\kappa}}{\sqrt{2(E-V_L)}}\\
				&=\int_{\kappa^-}^{\kappa^+}\pdv{\sqrt{2(E-V_L)}}{\kappa}\frac{\dd{\kappa}}{1+\kappa^0\kappa}\\
				&=\int_{\kappa^-}^{\kappa^+}\frac{\kappa^0\sqrt{2(E-V_L)}\dd{\kappa}}{\qty(1+\kappa^0\kappa)^2}.
			\end{align}
			and thus
			\begin{equation}
				C^-=\int_{\kappa^-}^{\kappa^+}\frac{\kappa^0\dd{\kappa}}{\qty(1+\kappa^0\kappa)^2\sqrt{2(E-V_L)}}=\int_0^{m\pi/n}\frac{\kappa^0\dd{\theta}}{\qty(1+\kappa^0\kappa)^2}>0.
			\end{equation}
			On the other hand, notice that integration by parts together with the identity
			\begin{equation}
				\qty(V_L-V^0_{L})\qty(\frac1\kappa+\kappa^0)=\kappa^0\frac{V_L-V^0_{L}}{\kappa-\kappa^0}\dv{V_L}{\kappa}
			\end{equation}
			leads to
			\begin{equation}
				\int_{\kappa^-}^{\kappa^+}\qty(V_L-V^0)\qty(\frac1\kappa+\kappa^0)\frac{\dd\kappa}{\sqrt{2(E-V_L)}}=\kappa^0\int_{\kappa^-}^{\kappa^+}\pdv{\kappa}\qty(\frac{V_L-V^0_{L}}{\kappa-\kappa^0})\sqrt{2(E-V_L)}\dd\kappa\label{leadsto}
			\end{equation}
			and thus
			\begin{align}
				&\qty(E-V^0_{L})\qty(\pdv{A_L}{E}+\kappa^0\Theta_{2\pi m/L}(E))\\
				&=\frac12\int_{\kappa^-}^{\kappa^+}\sqrt{2(E-V_L)}\qty(\frac1\kappa+\kappa^0)\dd\kappa+\kappa^0\int_{\kappa^-}^{\kappa^+}\pdv{\kappa}\qty(\frac{V_L-V^0_{L}}{\kappa-\kappa^0})\sqrt{2(E-V_L)}\dd\kappa.
			\end{align}
			Differentiating by $E$ and adding the vanishing term
			\begin{equation}
				\frac{1}{6(c+1)}\int_0^{m\pi/n}\dv{V_L}{\kappa}(\kappa)\dd{\theta}=-\frac{1}{6(c+1)}\eval{\kappa'}_0^{m\pi/n}=0
			\end{equation}
			a direct computation then shows
			\begin{equation}
				\qty(E-V^0_{L})C^+=\frac{1}{\kappa^0}\int_0^{m\pi/n}\frac{cf(u)+g(u)}{6(c+1)u(u-1)^2}\dd{\theta},
			\end{equation}
			where we have put $u\coloneqq\kappa/\kappa^0$ and $f(u)\coloneqq6u\log u+u^4-3u^3-u+3$ along with $g(u)\coloneqq6u\log u+u^3-6u^2+3u+2$. Observe that $f(1)=f'(1)=g(1)=g'(1)=0$, and
			\begin{equation}
				\dv[2]{f}{u}=\frac{6(u-1)^2(2u+1)}{u}\quad\text{and}\quad\dv[2]{g}{u}=\frac{6(u-1)^2}{u}
			\end{equation}
			are strictly positive on $(0,\infty)\setminus\{1\}$. Hence $f$ and $g$ are strictly convex with unique minimum $0$ at $u=1$, so both are positive on $(0,\infty)\setminus\{1\}$. As $\kappa\neq\kappa^0$ except at a single $\theta$ and $E>V^0_{2\pi m/L}$, we conclude $C^+>0$.
			Having shown $C^\pm>0$, we infer the result from \eqref{sign} in conjunction with \thref{montone}.
			\end{enumerate}
	\end{proof}
	\begin{coro}[fixed points of $\lambda\mapsto\bar{\kappa}(\lambda)$]
		For each $q\in W_\infty\cap \bbQ$, the map $\lambda\mapsto\bar{\kappa}(\lambda)$ has a unique fixed point in $\qty(\lambda_q,\infty)$.
	\end{coro}
	\begin{proof}
		We argue for the equivalent statement that $L\mapsto\len_q(L)$ has a unique fixed point in $(0,L_q)$. For $q<\sfrac{\sqrt{2}}{2}$ the result is evident from \thref{pro} \ref{pro1} and \ref{pro2}. For $q>\sfrac{\sqrt{2}}{2}$ we additionally recall \thref{pro} \ref{pro3} and notice that
		\begin{equation}
			\lim_{L\nearrow L_q}\len_q(L)=\frac{2\pi m L_q}{m\pi+\sqrt{m^2\pi^2 +L_q^2}}< L_q
		\end{equation}
		thus concluding the argument.
	\end{proof}
	\subsection{Abresch--Langer type curves}
	Now that we have established the existence of \eqref{APCSF}-shrinkers, we turn our attention to their saddle-point property, i.e., \thref{main-thmB}. First we introduce another class of curves originally proposed in \cite{WanKon14}.
	\begin{defi}[Abresch--Langer typology]
		Let $x:\bbS^1\to\bbR^2$ be an immersion. We say that $x$ (or $x(\bbS^1)$) is of \textbf{Abresch--Langer type} if
		\begin{enumerate}
			\item $x(\bbS^1)$ exhibits $n$-fold rotational symmetry with $n\in\bbN$, $n\geq 2$ and
			\begin{equation}
				\frac{m}{n}>\frac{1}{2}\label{ineq_al}
			\end{equation}
			with $m\in\bbN$ being the tangent turning index of $x$,
			\item $x$ is strictly locally convex,
			\item possibly after a rotation, the curvature and support function in tangential polar coordinates $\kappa,\sigma:\bbS^1\cong\bbR/2\pi m\bbZ\to\bbR$ are symmetric around $\theta^-=0$ and $\theta^+=\sfrac{\pi m}{n}$ and additionally monotone on $(0,\sfrac{\pi m}{n})$.
		\end{enumerate}
		Moreover, we define $\mathcal{A}_{m,n}$ as the set of all immersions for which the above conditions are true.
	\end{defi}
	Any Abresch--Langer curve is of course of Abresch--Langer type. More generally for $\lambda\geq 0$, by the discussion in Section \ref{part1}, any closed $\lambda$-curve with $\sfrac{m}{n}>\sfrac{1}{2}$ is also of Abresch--Langer type. Here $n,m\in\bbN$ is the rotational symmetry and tangent turning index (possibly after inverting the orientation of parametrization) of said $\lambda$-curve. There are some results available on this class of curves (e.g.~\cite{SesTsaWan20,AnaIshUsh25}). For our purposes a result due to \cite{WanKon14} that employs the sign of the enclosed area will be particularly helpful.
	\begin{theo}[Abresch--Langer type curves under \eqref{APCSF}; \cite{WanKon14}]
		Let $X:\bbS^1\times[0,T_\ast)$ be a maximal evolution by \eqref{APCSF} such that $X(\cdot,0)\in\mathcal{A}_{m,n}$ for some coprime $m,n\in\bbN$ satisfying $\sfrac{m}{n}>\sfrac{1}{2}$. Denote furthermore $\vol\coloneqq \vol[X(\cdot,0)]$ as the enclosed area along this evolution.
		\begin{enumerate}
			\item If $\vol>0$, then $T_\ast=\infty$ and $X(\cdot,t)$ converges to an $m$-fold cover of a circle as $t\to\infty$.
			\item If $\vol<0$, then $T_\ast<\infty$ and $X(\cdot,t)$ and $n$ cusps are formed as $t\nearrow T_\ast$.
		\end{enumerate}
	\end{theo}
	Now, proving \thref{main-thmB} simply amounts to arguing that normally perturbed \eqref{APCSF}-shrinkers are of Abresch--Langer type and verify the respective condition on the area stated above. This is established in the following final proposition.
	\begin{satz}[pertubations of $\lambda$-curves and \eqref{APCSF}-shrinkers]
		Let $\lambda\geq0$ and suppose that $x:\bbS^1\to\bbR^2$ is a positively oriented, non-circular $\lambda$-curve. Furthermore put $m,n\in \bbN$, $n\geq 2$, as the tangent turning index of $x$ and the degree of rotational symmetry of $x(\bbS^1)$. 
		\begin{enumerate}
			\item If $\varepsilon>0$ is sufficiently small, then $x^{\pm\varepsilon}\in\mathcal{A}_{m,n}$.\label{al1}
			\item If $x$ is additionally an \eqref{APCSF}-shrinker, then $\pm\vol[x^{\pm\varepsilon}]>0$.\label{al2}
		\end{enumerate}
	\end{satz}
	\begin{proof}
		For \ref{al1}, notice that by continuity $x^{\pm\varepsilon}$ and $x^{\pm\varepsilon}(\bbS^1)$ maintain the tangent turning index $m$ and the degree of rotational symmetry $n$ of $x$ and $x(\bbS^1)$. We show that \eqref{ineq_al} holds. To that end, we argue that the range $W_\lambda$ is bounded from below by $\frac{1}{2}$. We follow the argument by \cite{Cha17}. Recalling \thref{monotonicity_lambda} and \thref{ALcase}, we have
		\begin{equation}
			\frac{\pi}{2}=\lim\limits_{r\to\infty}\Theta_0(r)<\Theta_0(r)< \Theta_\lambda(r)
		\end{equation}
		for all $r>1$. Next, we consider the curvature $\kappa^{\pm\varepsilon}$ and the support function $\sigma^{\pm\varepsilon}$ of $x^{\pm\varepsilon}$. We find
		\begin{equation}
			\kappa^{\pm\varepsilon}=\frac{\kappa}{1\pm \varepsilon \kappa}\quad \text{and}\quad \sigma^{\pm\varepsilon}=\sigma\pm\varepsilon.\label{kappasigma}
		\end{equation}
		We see that $\abs{\kappa^{\pm\varepsilon}}>0$ if $\varepsilon>0$ is sufficiently small, so $x^{\pm\varepsilon}$ maintains the strict local convexity of $x$. The symmetry and monotonicity properies that are required of $\kappa^{\pm\varepsilon}$ and $\sigma^{\pm\varepsilon}$ are also easily verified from \eqref{kappasigma}. Finally, \ref{al2} can be easily observed from \eqref{area} in conjunction with \thref{zeroarea} or \thref{zeroarea2}.
	\end{proof}
	\bibliographystyle{alpha}
	\bibliography{main}
	\address{Goethe-Universit\"at Frankfurt am Main, Institut f\"ur Mathematik\\
		60325 Frankfurt, Germany\\
		\email{cernomaz@uni-frankfurt.de}
\end{document}